\documentclass[oneside,a4paper,12pt]{amsart} %revised 13 March 2012
\usepackage{amssymb}
\usepackage{amscd}
\usepackage[all,cmtip]{xy}
\usepackage{verbatim}
\usepackage{lscape}
\theoremstyle{plain}
\newtheorem{theorem}{Theorem}[section]
\newtheorem{lemma}[theorem]{Lemma}

\newtheorem{corollary}[theorem]{Corollary}
\newtheorem{proposition}[theorem]{Proposition}

\theoremstyle{definition}

\newtheorem{definition}[theorem]{Definition}

\theoremstyle{remark}
\newtheorem*{remark}{Remark}

\newtheorem*{example}{Example}
 \DeclareMathOperator{\Ext}{Ext}
\DeclareMathOperator{\Hom}{Hom} 
 
\DeclareMathOperator{\tr}{tr}

\DeclareMathOperator{\End}{End}

\DeclareMathOperator{\Ima}{Im}

\DeclareMathOperator{\proj}{proj} \DeclareMathOperator{\Id}{Id}
\DeclareMathOperator{\ind}{ind}
\DeclareMathOperator{\summ}{summ}

\newcommand{\Z}{{\mathbb{Z}}}

\newcommand{\F}{{\mathbb{F}}}

\begin{document}

\title{Exterior and Symmetric Powers of Modules for Cyclic 2-Groups}

\author{Frank Himstedt}
\address{Technische Universit\"at M\"unchen\\
         Zentrum Mathematik - M11\\
         Boltzmannstr.\ 3\\
         85748 Garching\\
         Germany}
\email{himstedt@ma.tum.de}

\author{Peter Symonds}
\address{School of Mathematics\\
         University of Manchester\\
     Manchester M13 9PL\\
     United Kingdom}
\email{Peter.Symonds@manchester.ac.uk}
\thanks{This project was supported by the Deutsche
  Forschungsgemeinschaft under the project KE 964/1-1
  (``Invariantentheorie endlicher und algebraischer Gruppen'').}
%\keywords{keywords}
%\subjclass{Primary: subject; Secondary: subject}
%\date{}

\begin{abstract}
We prove a recursive formula for the exterior and symmetric powers of
modules for a cyclic 2-group. This makes computation
straightforward. Previously, a complete description was 
only known for cyclic groups of prime order. 
\end{abstract}

\maketitle

%%%%%%%%%%%%%%%%%%%%%%%%%%%%%%%%%%%%%%%%%%%%%%%%%%%%%%%%%%%%%%%%%%%%%%%%%%%%%%%%

\section{Introduction}
\label{sec:intro}

The aim of this paper is to provide a recursive procedure for
calculating the exterior and symmetric powers of a modular
representation of a cyclic 2-group. Let $G \cong C_{2^n}$ be a cyclic
group of order $2^n$ and $k$ a field of characteristic $2$.
Recall that there are $2^n$ indecomposable $kG$-modules
$V_1,V_2, \ldots , V_{2^n}$ for which $\dim V_r=r$. 

\begin{theorem} \label{ext}
For all $n \ge 1$, $r \ge 0$ and $0 \leq s \leq 2^{n-1}$ we have
\[
\Lambda^r(V_{2^{n-1}+s}) \cong \bigoplus_{\genfrac{}{}{0pt}{1}{i,j \ge 0}{2i+j=r}} \Omega
_{2^n}^{i+j}( \Lambda ^i (V_s) \otimes_k \Lambda ^j (V_{2^{n-1}-s}))
\oplus t V_{2^n},
\]
where $t$ is a non-negative integer chosen so that both sides have the
same dimension.
\end{theorem}

Here $\Omega _{2^n}$ is the syzygy or Heller operator over $C_{2^n}$,
so $\Omega _{2^n} V_s = V_{2^n-s}$ for $1 \leq s \leq 2^{n}$.  
The group action on $V_1, \dots, V_{2^{n-1}}$ factors through
$C_{2^{n-1}}$ so that exterior powers of these modules can
be computed by applying the formula for this smaller group. In
particular, one can determine the exterior powers on the right hand
side of the formula in this way. We also show that there is a simple
recursive procedure for calculating tensor products. Since 
$\Lambda(A \oplus B) \cong \Lambda(A) \otimes \Lambda(B)$, we
obtain a complete recursive procedure for calculating exterior powers
of all possible modules. It is sufficiently efficient that it is easy
to calculate even by hand far beyond the range that was previously
attainable by machine computation.

For symmetric powers we use the following result from \cite{scyclic}.

\begin{theorem} \label{sym}
For all $n \ge 1$, $r \ge 0$ and $0 \leq s \leq 2^{n-1}$ we have
\[
S^r(V_{2^{n-1}+s}) \cong_{\ind} \Omega _{2^n}^{r'}
\Lambda^{r'}(V_{2^{n-1}-s}),
\]
where $0 \leq r' < 2^{n}$ and $r' \equiv r \pmod{2^{n}}$. Here the
symbol $\cong_{\ind}$ means up to direct summands induced from
subgroups $H \lvertneqq G$.
\end{theorem}

 Thus a knowledge of the
exterior powers determines the symmetric powers up to induced
summands. In fact it is shown in \cite{scyclic} how such a formula
determines the symmetric powers completely, using a recursive procedure.

Formulas for the exterior and symmetric powers of a module for a
cyclic group of prime order $p$ were given by Almkvist and Fossum
\cite{AlmkvistFossum} and Renaud~\cite{renaud2}. These were
extended~to cyclic p-groups by Hughes and Kemper
\cite{HughesKemperHilbert} provided that the power is at most $p-1$.
A formula for $\Lambda^2$ in the case of cyclic 2-groups was given by
Gow and Laffey~\cite{gowlaffey}. Also Kouwenhoven~\cite{kouwenhoven}
obtained important results on exterior powers of modules for cyclic
p-groups, including recursion formulas for $\Lambda(V_{q \pm 1})$
where $q$ is a power of $p$. For $p=2$ these formulas are special
cases or direct consequences of Theorem~\ref{ext}, so we obtain
independent proofs for some of the results in \cite{gowlaffey, kouwenhoven}.

Our strategy is to consider $\Lambda(V_{2^{n-1}+s})$ as the
quotient of $S(V_{2^{n-1}+s})$ by the ideal generated by
the squares of elements of $V_{2^{n-1}+s}$. It turns out that we need
to consider an intermediate ring $\tilde{S}(V_{2^{n-1}+s})$, in which
we only quotient out the squares of the elements of 
$V_s \subseteq V_{2^{n-1}+s}$. We show that 
$\tilde{S} ^r (V_{2^{n-1}+s}) \cong_{\ind} \Lambda ^r (V_{2^{n-1}+s})$
for $r < 2^n$.
But $\tilde{S}(V_{2^{n-1}+s})$ can be resolved by the Koszul
complex over $S(V_{2^{n-1}+s})$ on the squares of the elements of a basis for
$V_s$. We show that this Koszul complex is separated in the sense of
\cite{scyclic}, that is that the image of a boundary map is contained
in a projective submodule. This leads to the formula
\[
\tilde{S}^r(V_{2^{n-1}+s}) \cong_{\proj} \bigoplus _{2i+j=r} \Omega _{2^n}^i
( \Lambda ^i (V_s) \otimes_k S^j(V_{2^{n-1}+s})),
\]
where the symbol $\cong_{\proj}$ means up to projective summands.
Using Theorem \ref{sym}, the right hand side is easily seen to be
equal to the right hand side of the formula in Theorem~\ref{ext}
modulo induced summands. This yields the formula of Theorem~\ref{ext}
modulo induced summands. The strengthening to an equality modulo just
projective summands is a formal inductive argument.

We would like to thank Dikran Karagueuzian for the calculations that
were very helpful in discovering the formula of Theorem \ref{ext}.

%%%%%%%%%%%%%%%%%%%%%%%%%%%%%%%%%%%%%%%%%%%%%%%%%%%%%%%%%%%%%%%%%%%%%%%%%%%%%%%%

\section{Koszul Complexes}
\label{sec:koszul}

Let $G$ be a finite group, $H$ a subgroup of $G$ and $k$ a field
of characteristic $p>0$. All tensor products will be over $k$ if not
otherwise specified. We recall some general facts about chain
complexes of $kG$-modules from~\cite[Section~3]{scyclic}.

\begin{definition}(\cite[Definition~3.2]{scyclic}) \label{def:complex}
A chain complex $C_*$ of $kG$-modules is called:
\begin{enumerate}
\item[(a)] \textit{acyclic} if it is 0 in negative degrees and it only
  has homology in degree 0; 

\item[(b)] \textit{weakly induced from $H$} if each module is induced from $H$,
  and \textit{weakly induced from~$H$ except in degrees~$I$} if each $C_i$,
  $i \not \in I$, is induced from $H$;

\item[(c)] \textit{separated at $C_i$} if $\Ima(d_{i+1}) \rightarrow C_i$
  factors through a projective $kG$-module;

\item[(d)] \textit{separated} if it is separated at each $C_i$.
\end{enumerate}
\end{definition}

Write $B_i$ for $\Ima(d_{i+1}) \subseteq C_i$. If the inclusion
$B_i \hookrightarrow C_i$ factors through a projective then it factors
through the injective hull of $B_i$, call it $P_i$ (injective is
equivalent to projective for modular representations), and $P_i \hookrightarrow C_i$
is injective since it is so on the socle. Thus we can write
$C_i = P_i \oplus C_i'$ and $B_i \subseteq P_i$.

\begin{lemma}(\cite[Lemma~3.9]{scyclic})
\label{la:tensep}
If the chain complexes $C_*$, $C_*'$ of $kG$-modules are separated
then so is the (total) tensor product $C_* \otimes C_*'$. Similarly
for a product of finitely many chain complexes.
\end{lemma}

\begin{proof}
Let $P_i$ be a projective module such that
$\Ima (d_{i+1}) \subseteq P_i \subseteq C_i$ and similarly for
$P'_i$. Then, summing over all degrees,
$\Ima (d \otimes d') \subseteq P \otimes C' + C \otimes P'$ where 
$P := \bigoplus_i P_i$ and $P' := \bigoplus_i P'_i$. There is
a short exact sequence $0 \rightarrow P \otimes P' \rightarrow P
\otimes C' \oplus C \otimes P' \rightarrow P \otimes C + C \otimes P
\rightarrow 0$. The first two terms are projective, hence so is the third.
\end{proof}

We need to consider tensor-induced complexes. For details of the
construction see \cite[II 4.1]{benson}. 

\begin{lemma} \label{la:tenindsep}
Suppose that every elementary abelian p-subgroup of $G$ is conjugate
to a subgroup of $H$, and let $C_*$ be a complex of $kH$-modules that
is separated. Then the tensor-induced complex
${C_* \hspace{-0.1cm} \uparrow}^{\otimes G}_H$ is also separated.
\end{lemma}

\begin{proof}
By the proof of Lemma~\ref{la:tensep} above, the image
$\Ima ({d \hspace{-0.1cm} \uparrow}^{\otimes G}_H)$ is contained in
$$P \otimes C \otimes \cdots \otimes C + C \otimes P \otimes \cdots
\otimes C + \cdots + C \otimes C \otimes \cdots \otimes P,$$ which is a
$kG$-submodule of ${C_* \hspace{-0.1cm} \uparrow}^{\otimes G}_H$. But
the same proof shows that this module is projective on restriction to
$H$, so it is projective, by Chouinard's
Theorem~\cite[Corollary~1.1]{chouinard}.
\end{proof}

The next two results comprise a variation on
\cite[Proposition~3.3]{scyclic} and have the same proof.

\begin{proposition} \label{prop:homol2}
Let $H$ be an arbitrary subgroup of $G$. Suppose that the complex
$K_*: K_w \rightarrow \cdots \rightarrow K_0$ of $kG$-modules is: 
\begin{enumerate}
\item[(a)] acyclic,

\item[(b)] weakly induced from $H$ except in at most one degree and

\item[(c)] $K_*$ is separated on restriction to $H$.
\end{enumerate}
Then $K_*$ is separated.
\end{proposition}

Recall that the Heller translate $\Omega V$ of a $kG$-module $V$ is
defined to be the kernel of the projective cover $P(V) \rightarrow V$
and $\Omega^i V$ for $i \ge 1$ denotes $\Omega$ iterated $i$
times. Similarly $\Omega^{-1}V$ is the cokernel of the injective hull
$V \rightarrow I(V)$ and $\Omega^{-i}V$ for $i \ge 1$ is its
iteration. We let $\Omega^0 V$ denote~$V$ with any projective summands
removed. These have the properties that
$\Omega^i \Omega^j V \cong \Omega^{i+j} V$ and that if $V$ is induced so is
$\Omega^i V$.

\begin{lemma} \label{la:homol1}
Suppose that the complex
$K_*: K_w \rightarrow \cdots \rightarrow K_0$ of $kG$-modules is:
\begin{enumerate}
\item[(a)] acyclic with $H_0(K_*) = L$, and

\item[(b)] separated.
\end{enumerate}
Then $L \cong_{\proj} K_0 \oplus \Omega^{-1}K_1 \oplus \Omega^{-2}K_2
\oplus \cdots \oplus \Omega^{-w}K_w$.
\end{lemma}

Let $V$ be a $kG$-module, finite-dimensional
as a $k$-vector space, and $W$ a submodule of $V$. We write
$S = S(V) = \bigoplus_{r=0}^\infty S^r(V)$ for the symmetric algebra
on $V$ and $\Lambda(W) = \bigoplus_{r=0}^\infty \Lambda^r(W)$ for
the exterior algebra on $W$. For $r<0$ let $S^r(V)$ denote the $0$
module.

\begin{definition} \label{def:koszul}
Let $W$ be a submodule of a $kG$-module $V$ and let $W^p$ denote the
$kG$-submodule of $S^p(V)$ spanned by the $p$-th powers of elements of
$W$. Let $K(V, W^p)$ denote the Koszul complex of graded $kG$-modules: 
\[
\dots \stackrel{d}{\longrightarrow} S(V) \otimes \Lambda^3(W^p)
\stackrel{d}{\longrightarrow} S(V) \otimes \Lambda^2(W^p)
\stackrel{d}{\longrightarrow} S(V) \otimes W^p
\stackrel{d}{\longrightarrow} S(V),
\]
where $d(f \otimes w_1^p \wedge w_2^p \wedge \dots \wedge w_i^p) =
\sum_{j=1}^i (-1)^{j+1} fw_j^p \otimes w_1^p \wedge \dots \wedge \widehat{w_j^p} \wedge
\dots \wedge w_i^p$ for $w_j \in W$ and $f \in S(V)$.
We write $K^r(V, W^p)$ when we consider the complex $K(V, W^p)$ in
grading-degree~$r$.
\end{definition}

If $k=\mathbb F_2$ then the squaring map gives an isomorphism between
$W$ and $W^2$, so we can regard $W^2$ as a copy of $W$ in degree 2
equipped with a squaring map into $S^2(V)$. From this point of view,
the boundary map is given by 
 $d(f \otimes w_1 \wedge w_2 \wedge \dots \wedge w_i) = \sum_{j=1}^i
  fw_j^2 \otimes w_1 \wedge \dots \wedge \widehat{w_j} \wedge
\dots \wedge w_i$ for $w_j \in W$ and $f \in S(V)$.

We will normally take the second point of view, so we will assume that
$k=\mathbb F_2$ in a large part of this paper. Since any
$kC_{2^n}$-module can be written in $\mathbb F_2$, this is not a
significant restriction. 

\begin{lemma} \label{la:KVWhom}
In the context of
Definition~\ref{def:koszul}, the complex $K(V, W^p)$ is acyclic and
its homology in degree~0 is $S(V)/(W^p)$, where $(W^p)$ is the ideal
generated by all elements $w^p$, $w \in W$.
\end{lemma}

\begin{proof}
If $\{ w_1, \ldots , w_r \}$ is a basis for $W$ then 
$\{ w_1^p, \ldots, w_r^p \}$ is a regular sequence in $S(V)$ and it
spans $W^p$. This is now a standard result about Koszul complexes. 
\end{proof}

\begin{lemma} \label{la:dirsum}
Let $V$, $V'$ be $kG$-modules, finite-dimensional as $k$-vector spaces
and let $W$, $W'$ be submodules of $V$ and $V'$, respectively. The
complex $K(V \oplus V', (W \oplus W')^p)$ is isomorphic to the
(total) tensor product $K(V, W^p) \otimes K(V', W'^p)$ as a
complex of graded $kG$-modules.
\end{lemma}

\begin{proof}
This is analogous to~\cite[Lemma~3.8]{scyclic}.
\end{proof}

We also need to deal with tensor induction of graded modules and
complexes.

\begin{lemma} \label{la:tensind}
Let $H$ be a subgroup of $G$ and let $V, W$ be
$kH$-modules. Then $S(V \hspace{-0.1cm} \uparrow^G_H)
\cong S(V) \hspace{-0.1cm} \uparrow^{\otimes G}_H$,
$\Lambda(V \hspace{-0.1cm} \uparrow^G_H) \cong
\Lambda(V) \hspace{-0.1cm} \uparrow^{\otimes G}_H$ as graded
$kG$-modules, and if the characteristic of $k$ is 2 
then $K(V \hspace{-0.1cm} \uparrow^G_H, (W \hspace{-0.1cm}
\uparrow^G_H)^2) \cong K(V,W^2) \hspace{-0.1cm} \uparrow^{\otimes
  G}_H$ as complexes of graded $kG$-modules. 
\end{lemma}

Without the restriction on the characteristic of $k$ we would have to
deal with the sign convention that appears in the definition of the
action of $G$ on the tensor-induced complex. 

\begin{proof}
Let $\{ g_i \}$ be a set of coset representatives for $G/H$ and write
$V \hspace{-0.1cm} \uparrow ^G_H = \oplus g_i \otimes V$. The formulas
now follow from the usual formulas for $S$ and $\Lambda$ of a sum and
the definition of the group action on a tensor induced module. 
\end{proof}

%%%%%%%%%%%%%%%%%%%%%%%%%%%%%%%%%%%%%%%%%%%%%%%%%%%%%%%%%%%%%%%%%%%%%%%%%%%%%%%%

\section{Modules for Cyclic 2-Groups}
\label{sec:cyc2}

From now on, let $G = \langle g \rangle \cong C_{2^n}$ be a cyclic
group of order $2^n$, $n \ge 1$, and $k$ a field of characteristic~$2$.
We write $a(G)$ for the Green ring of $kG$-modules. Up to isomorphism,
there are~$2^n$ indecomposable $kG$-modules $V_1$, $V_2$, \dots,
$V_{2^n}$ and we choose the notation so that $\dim_k(V_i) = i$. For
convenience we write~$V_0$ for the $0$ module. The generator $g \in G$
acts on~$V_i$ with matrix a Jordan block with ones on the diagonal.
Choose a $k$-basis $\{x_1, x_2, \dots, x_n\}$ of~$V_n$ such that
$g x_i = x_i + x_{i-1}$ for all $2 \le i \le n$ and $g x_1 =
x_1$. Each element of $S(V_i)$ can be written uniquely as a polynomial
in $x_1$, \dots, $x_i$, and for $j \le i$, we can identify $V_j$
with the $kG$-submodule of $V_i$ spanned by $x_1, x_2, \dots, x_j$. Each
$V_i$ is uniserial with composition series
$0 < V_1 < V_2 < \dots < V_{i-1} < V_i$. Note that for $i \le 2^{n-1}$
the kernel of~$V_i$ is nontrivial and so~$V_i$ can be identified with
the $i$-dimensional indecomposable module for the quotient group~$C_{2^{n-1}}$.

Decompositions of tensor products into indecomposables have been studied
by several authors, see for
example~\cite{barry1,barry2,green,hou,norman1,norman2,renaud,srinivasan}. In
our case, this decomposition can easily be computed using the Heller
translate $\Omega$. We write $\Omega_{2^n}$ instead of $\Omega$ when
we want to emphasize that we are working with modules for the group
$C_{2^n}$. It is easy to check that $\Omega _{2^n} V_i \cong _{\proj}
V_{2^n-i}$ for $0 \leq i \leq 2^n$, where $\proj$ means modulo
projective modules for $C_{2^n}$. 

Recall that
$\Omega_{2^n}(V \otimes V') \cong _{\proj} (\Omega_{2^n} V) \otimes
V'$, where the projective part can be determined by comparing
dimensions. For cyclic groups, $\Omega_{2^n}^2 \cong _{\proj} \Id$ and
$\Omega_{2^n}V \otimes \Omega_{2^n}V' \cong _{\proj} V \otimes
V'$. This provides an easy recursive method for calculating the 
decomposition of tensor products in the case of cyclic 2-groups.

In order to calculate $V_a \otimes V_b$, we may assume $a \geq b$ and
write $a= 2^r-a'$ for the smallest possible $r$ such that $a' \geq
0$. Then $V_a \otimes V_b \cong (\Omega_{2^r} V_{a'}) \otimes V_b
\cong \Omega _{2^r} (V_{a'} \otimes V_b)$ modulo copies
of~$V_{2^r}$. If $b \geq 2^{r-1}$ then it is more efficient to write 
$b=2^r-b'$ too. 

\begin{example}
For $V_{9}, V_{13} \in a(C_{16})$ we have:
$V_{9} \otimes V_{13} \cong (\Omega_{16} V_7) \otimes (\Omega_{16}
V_3) \, \cong \, (V_7 \otimes V_3) \oplus$ copies of~$V_{16}$. By
comparing dimensions we get $V_{9} \otimes V_{13} \cong (V_7 \otimes
V_3) \oplus 6 V_{16}$. Now we consider the non-faithful module $V_7
\otimes V_3$ as a module for the factor group $\cong C_8$ and get $V_7
\otimes V_3 \cong \Omega_8(V_1 \otimes V_3) \oplus$ copies of
$V_8$. Again by comparing dimensions we obtain $V_7 \otimes V_3 \cong
\Omega_8(V_1 \otimes V_3) \oplus 2 V_8 \cong V_5 \oplus 2 V_8$ and
hence $V_{9} \otimes V_{13} \cong V_5 \oplus 2 V_8 \oplus 6 V_{16}$.
\end{example}

%% \begin{example}
%% For $V_{9}, V_{13} \in a(C_{16})$ we have:
%% \begin{eqnarray*}
%% V_{9} \otimes V_{13} & \cong & (\Omega_{16} V_7) \otimes (\Omega_{16}
%% V_3) \, \cong \, (V_7 \otimes V_3) \oplus 6 V_{16} \, \cong \, ((\Omega _8 V_1) \otimes V_3) \oplus 6 V_{16}
%% \\ & \cong & \Omega_8(V_1 \otimes V_3) \oplus 2 V_8 \oplus 6 V_{16} \, \cong \,
%%   (\Omega_8 V_3) \oplus 2 V_8 \oplus 6 V_{16}
%% \\ & \cong  & V_5 \oplus 2 V_8 \oplus 6 V_{16}.
%% \end{eqnarray*}
%% \end{example}

Let $H = \langle g^2 \rangle$ be the unique maximal subgroup of
$G$. For $1 \le j \le 2^{n-1}$ we also denote the indecomposable
$kH$-module of dimension $j$ by $V_j$. Of course, this is an
abuse of notation, but we will always make it clear whether we consider
$V_j$ as a $kG$-module or as a $kH$-module. An elementary calculation
with Jordan canonical forms shows that the restriction operator
$\downarrow^G_H: a(G) \rightarrow a(H)$ is given by
$V_i \hspace{-0.1cm} \downarrow^G_H = V_{i'} \oplus V_{i''}$ where
$V_{i'}$ is the $kH$-submodule generated by
$\{x_i, x_{i-2}, x_{i-4}, \dots\}$ and
$V_{i''}$ is the $kH$-submodule generated by $\{x_{i-1}, x_{i-3},
x_{i-5}, \dots\}$. In particular, we have $(i', i'') = (\frac{i+1}{2},
\frac{i-1}{2})$ if $i$ is odd, and $(i', i'') = (\frac{i}{2},
\frac{i}{2})$ if $i$ is even. The induction operator
$\uparrow^G_H: a(H) \rightarrow a(G)$ is given
by $V_j \hspace{-0.1cm} \uparrow^G_H = V_{2j}$ for $1 \le j \le 2^{n-1}$.

We say that a $kG$-module is induced if it is induced from proper
subgroups. Let $a_P(G)$ be the submodule of $a(G)$ generated by the
projective modules and $a_I(G)$ the submodule generated by the induced
modules. Notice that $a_P(G)$ and $a_I(G)$ are ideals of $a(G)$ and
that induction maps $a_P(H)$ into $a_P(G)$ and $a_I(H)$ into $a_I(G)$,
but restriction only maps $a_P(G)$ into $a_P(H)$.

The following lemmas deduce information on $kG$-modules and short
exact sequences of $kG$-modules from their restriction to $H$.

\begin{lemma} \label{la:resind}
Let $A$ be a $kG$-module such that $A \hspace{-0.1cm} \downarrow^G_H$
is induced from a proper subgroup of~$H$. Then $A$ is induced from $H$.
\end{lemma}

\begin{proof}
We can assume that $A$ is indecomposable. Since
$A \hspace{-0.1cm} \downarrow^G_H$ is induced, each indecomposable
direct summand of $A \hspace{-0.1cm} \downarrow^G_H$ has even
dimension. Thus $\dim_k(A)$ is even and so $A$ is induced.
\end{proof}

As in the introduction we write $\cong_{\ind}$ and $\cong_{\proj}$
for isomorphisms modulo induced and modulo projective summands,
respectively. 

\begin{lemma} \label{la:ind}
Let $A$, $B$ be induced $kG$-modules such that $A \hspace{-0.1cm}
\downarrow^G_H \cong_{\proj} B \hspace{-0.1cm} \downarrow^G_H$. Then
$A \cong_{\proj} B$.
\end{lemma}

\begin{proof}
Since $A$ is induced,
$A \hspace{-0.1cm} \downarrow^G_H \uparrow ^G_H \cong 2A$; the same is
true for $B$. Inducing a projective yields a projective, so we obtain
$2A \cong_{\proj} 2B$ and the claim follows.
\end{proof}

\begin{lemma} \label{la:modind}
Let $A$ and $B$ be $kG$-modules such that $A \cong_{\ind} B$ and
$A \hspace{-0.1cm} \downarrow^G_H \, \cong_{\proj} B \hspace{-0.1cm}
\downarrow^G_H$. Then $A \cong_{\proj} B$.
\end{lemma}

\begin{proof}
We have $A \oplus X \cong B \oplus Y$ for some induced modules $X$ and
$Y$. On restriction, we obtain $X \hspace{-0.1cm} \downarrow^G_H
\cong_{\proj} Y \hspace{-0.1cm} \downarrow^G_H$ and so $X
\cong_{\proj} Y$ by Lemma~\ref{la:ind}. Now cancel the non-projective
summands of $X$ and $Y$ in the original formula.
\end{proof}

\begin{lemma}
\label{la:looks_sep}
Let $0 \rightarrow A \rightarrow B \rightarrow C \rightarrow 0$ be a
short exact sequence of $kG$-modules that is separated at $B$ on restriction
to $H$ and such that $C \cong _{\ind} B \oplus \Omega_{2^n}^{-1}A$ as
$kG$-modules. Then the sequence is separated at $B$ (as a sequence of
$kG$-modules).
\end{lemma}

\begin{proof}
The hypotheses imply that there are induced modules $X$ and $Y$ such
that $C \oplus X \cong B \oplus \Omega_{2^n}^{-1}A \oplus Y$ and also
${C \hspace{-0.1cm} \downarrow^G_H} \cong_{\proj}
{B \hspace{-0.1cm} \downarrow^G_H} \oplus {\Omega_{2^{n-1}}^{-1}A
\hspace{-0.1cm} \downarrow^G_H} \cong_{\proj}
{(B \oplus \Omega_{2^n}^{-1}A) \hspace{-0.1cm} \downarrow^G_H}$ by
Lemma~\ref{la:homol1} applied to $0 \rightarrow {A \hspace{-0.1cm}
\downarrow^G_H} \rightarrow {B \hspace{-0.1cm} \downarrow^G_H} \rightarrow 0$.
It follows that ${X \hspace{-0.1cm} \downarrow^G_H} \cong_{\proj}
{Y \hspace{-0.1cm} \downarrow^G_H}$, hence, by Lemma~\ref{la:ind},
$X \cong_{\proj} Y$ and then $C \cong _{\proj} B \oplus \Omega_{2^n}^{-1}A$.

Thus our short exact sequence is $0 \rightarrow A
\stackrel{d}{\rightarrow} B \stackrel{e}{\rightarrow} B \oplus
\Omega_{2^n}^{-1}A \rightarrow 0$, up to projective summands.
Consider the long exact sequence for Tate $\Ext$:
\begin{equation} \label{eq:longtate}
\cdots \rightarrow \overline{\Hom}_{kG}(A,A) \stackrel{d_*}{\rightarrow}
\overline{\Hom}_{kG}(A,B) \stackrel{e_*}{\rightarrow} \overline{\Hom}_{kG}(A,B \oplus
\Omega_{2^n}^{-1}A) \rightarrow \Ext^1(A,A) \rightarrow \cdots ,
\end{equation}
where $\overline{\Hom}_{kG}$ denotes homomorphisms modulo those that
factorize through a projective.

Since $\Ext^1(A,A) \cong \overline{\Hom}_{kG}(A, \Omega_{2^n}^{-1}A)$
we have 
\begin{eqnarray*}
\dim \Ima(e_*) + \dim \Ext^1(A,A) & \ge &
\dim \overline{\Hom}_{kG}(A,B \oplus \Omega_{2^n}^{-1}A) = \\
&& \dim \overline{\Hom}_{kG}(A,B) + \dim \overline{\Hom}_{kG}(A, \Omega_{2^n}^{-1}A) = \\
&& \dim \overline{\Hom}_{kG}(A,B) + \dim \Ext^1(A,A)
\end{eqnarray*}
and therefore $\dim \Ima(e_*) \ge \dim \overline{\Hom}_{kG}(A,B)$. 
Hence $e_*$ is injective and so $d_*=0$.
But $d=d_*(\Id_A)$, so $d$ factors through a projective, as required.
\end{proof}

The next lemma describes tensor induction from $H$ to~$G$ modulo
induced modules and gives information on the structure of the exterior
algebra $\Lambda(V_{2j})$ as a $kG$-module in terms of the $kH$-module
$\Lambda(V_j)$.

\begin{lemma} \label{la:ten_ind}
Let $r$, $j$ be integers such that $r \ge 0$ and $1 \le j \le
2^{n-1}$. We consider $V_j$ as a $kH$-module and
$V_{2j} = V_j \hspace{-0.1cm} \uparrow^G_H$ as a $kG$-module.

\begin{enumerate}
\item[(a)] Let $A$ and $B$ be $kH$-modules. Then $(A \oplus
  B) \hspace{-0.1cm} \uparrow^{\otimes G}_H \cong A \hspace{-0.1cm}
  \uparrow^{\otimes G}_H \oplus B \hspace{-0.1cm} \uparrow^{\otimes
    G}_H \oplus X$ for some induced $kG$-module $X$.

\item[(b)] There is an induced $kG$-module $X$ such that
  $\Lambda^{2r}(V_{2j}) \cong {\Lambda^r(V_j) \hspace{-0.1cm}
  \uparrow_H^{\otimes G}} \oplus X$.

\item[(c)] If $r$ is odd, then the $kG$-module $\Lambda^r(V_{2j})$ is
  induced from $H$.

\item[(d)] If $j$ is even, then the $kG$-module
  $V_j \hspace{-0.1cm} \uparrow_H^{\otimes G}$ is induced from $H$.

\item[(e)] If $j$ is odd, then ${V_j \hspace{-0.1cm}
  \uparrow_H^{\otimes G}} \, \cong_{\ind} V_1$.
\end{enumerate}
\end{lemma}

\begin{proof}
(a) follows from \cite[I 3.15.2 (iii)]{benson}.

\medskip

\noindent (b) By the construction of induced modules, we have
$V_{2j} = V_j \oplus g V_j$ as vector spaces and the action of the
  generator $g$ of $G$ on $V_{2j}$ is given by
$g (v + g v') = g^2 v' + g v$. So there is a natural isomorphism
\begin{eqnarray*}
\Lambda^{2r}(V_{2j}) & = & \Lambda^{2r}(V_j \oplus g V_j) \cong
\bigoplus_{\genfrac{}{}{0pt}{}{r',r'' \ge 0}{r'+r''=2r}}
(\Lambda^{r'}(V_j) \otimes \Lambda^{r''}(g V_j))
\end{eqnarray*}
of vector spaces, and thus
\[
\Lambda^{2r}(V_{2j}) \cong
(\Lambda^r(V_j) \otimes g \Lambda^r(V_j)) \oplus
\bigoplus_{\genfrac{}{}{0pt}{}{0 \le r' < r''}{r'+r''=2r}}
((\Lambda^{r'}(V_j) \otimes \Lambda^{r''}(g V_j)) \oplus
g (\Lambda^{r'}(V_j) \otimes \Lambda^{r''}(g V_j))).
\]
Via this isomorphism, the right hand side becomes a $kG$-module and
from the action of $g$, we see that
$(\Lambda^{r'}(V_j) \otimes \Lambda^{r''}(g V_j)) \oplus
g (\Lambda^{r'}(V_j) \otimes \Lambda^{r''}(g V_j))$ is a
$kG$-submodule isomorphic to
$(\Lambda^{r'}(V_j) \otimes \Lambda^{r''}(g V_j)) \hspace{-0.1cm}
\uparrow^G_H$ and $\Lambda^r(V_j) \otimes g \Lambda^r(V_j)$ is a
submodule isomorphic to ${\Lambda^r(V_j) \hspace{-0.1cm} \uparrow_H^{\otimes G}}$.

\medskip

\noindent (c) The proof is similar to that of (b). Note that, if $r$ is odd,
the summand corresponding to $r'=r''$ which leads to the tensor
induced submodule in (b) does not occur.

\medskip

\noindent (d),(e) We say that a $kG$-module is \textit{induced except
for possibly one trivial summand} if it is isomorphic to
${A \hspace{-0.1cm} \uparrow^G_H}$ or
${A \hspace{-0.1cm} \uparrow^G_H} \oplus V_1$ for some $kH$-module $A$.
We prove (d) and (e) simultaneously by showing that for all
$1 \le j \le 2^{n-1}$ the $kG$-module
${V_j \hspace{-0.1cm} \uparrow_H^{\otimes G}}$ is induced except for
possibly one trivial summand. The claim then follows from the fact
that $\dim_k({V_j \hspace{-0.1cm} \uparrow_H^{\otimes G}})$ is
even if and only if $j$ is even.

The proof is by induction on $j$. Because
${V_1 \hspace{-0.1cm} \uparrow_H^{\otimes G}} \cong V_1$ we can assume
$j>1$. If $j$ is even, then the $kH$-module $V_j$ is induced from a
proper subgroup of $H$. So~\cite[I 3.15.2 (iv)]{benson} implies that
${V_j \hspace{-0.1cm} \uparrow_H^{\otimes G}}$ is a direct sum of
modules induced from~$H$ (even from proper subgroups of $H$). Assume
that $j$ is odd. So we can write $j = 2^m + j'$ with $1 \le m < n-1$
and $1 \le j' < 2^m$. First, we treat the case $j'=1$. By the
Mackey formula for tensor induction \cite[I 3.15.2 (v)]{benson} we have
$V_j \hspace{-0.1cm} \uparrow^{\otimes G}_H \downarrow^G_H \cong
V_{2^m+1} \otimes V_{2^m+1} \cong V_1 \oplus (2^m-2) V_{2^m} \oplus 2
V_{2^{m+1}}$, and so ${V_j \hspace{-0.1cm} \uparrow^{\otimes G}_H}
\cong V_1 \oplus (2^{m-1}-1) V_{2^{m+1}} \oplus V_{2^{m+2}}$,
which is induced up to one trivial summand. Now assume $j' > 1$.
Then $V_{j'} \otimes V_{2^m+1} \cong V_j \oplus (j'-1) V_{2^m}$ as
$kH$-modules. By~\cite[I 3.15.2~(i)]{benson} and~(a) we get
\begin{equation} \label{eq:tensin}
({V_{j'} \hspace{-0.1cm} \uparrow^{\otimes G}_H}) \otimes
  ({V_{2^m+1} \hspace{-0.1cm} \uparrow^{\otimes G}_H}) \cong_{\ind}
({V_j \hspace{-0.1cm} \uparrow^{\otimes G}_H}) \oplus (j'-1)
  ({V_{2^m} \hspace{-0.1cm}  \uparrow^{\otimes G}_H}).
\end{equation}
By induction and the case $j'=1$, we know that the left hand side
of~(\ref{eq:tensin}) is induced except for possibly one trivial
summand. Hence, ${V_j \hspace{-0.1cm} \uparrow^{\otimes G}_H}$ is
induced except for possibly one trivial summand.
\end{proof}

We can now see that the symmetric and exterior powers of even
dimensional indecomposable modules have a particularly restricted
form.

\begin{corollary} \label{cor:lambdaeven}
Suppose that we have non-negative integers $j,s,t,u$ with $u,j$ odd
and $s \ge 1$. Furthermore, assume that $2^tu < 2^n$ and
$2^sj \le 2^n$. Then $\Lambda^{2^tu}(V_{2^sj})$ and
$S^{2^tu}(V_{2^sj})$ are both induced unless $t \ge s$. If $t \ge s$
then $\Lambda^{2^tu}(V_{2^sj}) \cong m V_1 \oplus X$ and
$S^{2^tu}(V_{2^sj}) \cong m' V_1 \oplus Y$, where $X, Y$ are induced
modules and $m$ and $m'$ are the numbers of non-induced indecomposable
summands in $\Lambda^{2^{t-s}u}(V_j)$ and
$\Lambda^{2^{t-s}u}(V_{2^{n-s}-j})$, respectively.
\end{corollary}

\begin{proof}
Using Lemma~\ref{la:ten_ind} (a),(b),(d) we see that, up to induced
direct summands, $\Lambda^{2^tu}(V_{2^sj})$ is tensor-induced from a
subgroup of index $2^{\min \{s,t\}}$. If $t < s$ then, up to induced
direct summands, it is tensor-induced from $\Lambda^u (V_{2^{s-t}j})$
and thus is induced, by part (c) of the same lemma. If $t \ge s$ then,
again up to induced direct summands, it is tensor-induced from
$\Lambda^{2^{t-s}u} (V_j)$; the description given is then seen to be
valid using parts (a), (d) and (e). The case of $S^{2^tu}(V_{2^sj})$
reduces to that of $\Lambda^{2^tu}(V_{2^n-2^sj})$, by Theorem~\ref{sym}.
\end{proof}

\begin{corollary}
\label{cor:s2}
If $X$ is a $kG$-module such that every direct summand has dimension
divisible by~4 then $S^2(X)$ is induced.
\end{corollary}

\begin{proof}
By the identity
$S^2(A \oplus B) \cong S^2(A) \oplus S^2(B) \oplus A \otimes B$, we
may assume that $X$ is indecomposable, say $X=V_{4u}$. The claim now
follows from Corollary~\ref{cor:lambdaeven}.
\end{proof}

In the proof of our main result we will often have information only
modulo induced direct summands. The following definition and lemmas
deal with the splitting of maps in such situations.

Recall that a map $f: A \rightarrow B$ of $kG$-modules is
\textit{split injective}, if there is a map $g: B \rightarrow A$ of $kG$-modules
such that $g \circ f = \Id_A$. For maps $f: A \rightarrow B$,
$f': A \rightarrow B'$ of $kG$-modules we write
$(f,g): A \rightarrow B \oplus B'$, $a \mapsto (f(a), f'(a))$.

\begin{definition} \label{def:splitmi}
Let $f:A \rightarrow B$ be a map of $kG$-modules. We say that
$f$ is \textit{split injective modulo induced summands} if there exists an
induced $kG$-module $X$ and a map $f':A \rightarrow X$ of
$kG$-modules such that $(f,f'): A \rightarrow B \oplus X$ is split injective.
\end{definition}

Split injective modulo induced summands behaves in much the same way
as split injective. 

\begin{lemma}
\label{la:prop_split}
Given maps $f: A \rightarrow B$, $g: B \rightarrow C$ and $h: D
\rightarrow E$ of $kG$-modules:
\begin{enumerate}
\item[(a)]
if $f$ and $g$ are split injective modulo induced summands then so is $g \circ f$,
\item[(b)]
if $g \circ f$ is split injective modulo induced summands then so is $f$,
\item[(c)]
if $f$, $h$ are split injective modulo induced summands then so is
$f \otimes h: A \otimes D \rightarrow B \otimes E$.
\end{enumerate}
\end{lemma}

\begin{proof}
(a) By assumption, we have induced modules $X$, $Y$ and maps
$f': A \rightarrow X$, $u: B \rightarrow A$, $u': X \rightarrow A$,
  $g': B \rightarrow Y$, $v: C \rightarrow B$, $v': Y \rightarrow B$
such that $u \circ f + u' \circ f'=\Id_A$ and
$v\circ g + v' \circ g'=\Id_B$. We define $(g \circ f)': A \rightarrow
X \oplus Y, a \mapsto (f'(a), g' \circ f(a))$,
$w: C \rightarrow A, c \mapsto u \circ v(c)$ and
$w': X \oplus Y \rightarrow A, (x,y) \mapsto u'(x)+u \circ v'(y)$.
Then $w \circ (g \circ f) + w' \circ (g \circ f)' = \Id_A$.

Parts (b) and (c) are proved in a similar way; the proofs are left to the reader.
\end{proof}

\begin{lemma}
\label{la:split_mod_ind}
Let $f: A \rightarrow B$ be a map of $kG$-modules and write
$A = A' \oplus A''$, where $A'$ has only non-induced summands and
$A''$ has only induced summands. Let $i$ denote the inclusion of $A'$
in $A$. Then $f$ is split injective modulo induced summands if and only if
$f \circ i$ is split injective.
\end{lemma}

\begin{proof}
Suppose that $f$ is split injective modulo induced summands; we want to show
that $f \circ i$ is split injective. By Lemma~\ref{la:prop_split} (a), the map
$f \circ i$ is split injective modulo induced summands, so we can assume that
$A = A'$ and we have to show that $f$ is split injective.

Since $f$ is split injective modulo
induced summands we have an induced module $X$ and maps
$f': A \rightarrow X$, $u: B \rightarrow A$, $u': X \rightarrow A$
such that $u \circ f + u' \circ f' = \Id_A$. Since $X$ and $A$ have no
summands in common, we know that $u' \circ f'$ lies in the radical of
$\End_{kG}(A)$ (note that if $A = \bigoplus A_i$ with $A_i$
indecomposable and we write elements of $\End_{kG}(A)$ as matrices
with entries in $\Hom_{kG}(A_i,A_j)$ then the radical consists of the
morphisms for which no component is an isomorphism). Thus $u \circ f$
is surjective, hence an automorphism of $A$, and $f$ is split injective. 

Conversely, suppose that $f \circ i: A' \rightarrow B$ is split injective, so
there is a map $g: B \rightarrow A'$ such that
$g \circ (f \circ i) = \Id_{A'}$. Let $j$ denote the inclusion of
$X := A''$ in $A$ and $f'$ the projection of $A$ onto~$A''$. We define
$v := i \circ g: B \rightarrow A$, and
$v': X \rightarrow A, \; x \mapsto -(i \circ g \circ f \circ j)(x) + j(x)$.
Then $v \circ f + v' \circ f' = \Id_A$, so $f$ is split injective modulo induced
summands.
\end{proof}

\begin{remark}
The proof above shows that the induced module $X$ in
Definition~\ref{def:splitmi} can always be chosen in such a way
that~$X$ only contains indecomposable direct summands that also occur in $A$.
\end{remark}

\begin{remark}
Definition~\ref{def:splitmi} makes sense for any finite group and any
class of indecomposable modules and Lemmas~\ref{la:prop_split}(a,b)
and \ref{la:split_mod_ind} remain true. 
\end{remark}

It will turn out that certain symmetric and exterior powers of modules
for cyclic 2-groups are contained in the $\mathbb Z$-submodule $c(G)$
of the Green ring $a(G)$ spanned by the indecomposable modules $V_r$
for~$r$ satisfying $r \not \equiv 2 \pmod{4}$. We describe some
properties of $c(G)$.

\begin{lemma} \label{la:c}
The submodule $c(G)$ is
\begin{enumerate}
\item[(a)] a subring of $a(G)$ and
\item[(b)] closed under $\Omega_{2^n}$.
\end{enumerate}
\end{lemma}

\begin{proof}
Part (b) is clear from the definitions.

For part (a) we need to show that $V_i \otimes V_j \in c(G)$ for all
$0 \le i, j \le 2^n$, $i,j \not \equiv 2 \pmod{4}$. For $n=1$ we only
have $V_1 \otimes V_1 = V_1 \in c(G)$. Suppose that $n > 1$. By the
remarks on the computation of tensor products at the beginning of this
section, we have $V_i \otimes V_j = \Omega_{2^n}^m (V_{i'} \otimes
V_{j'}) \oplus m' V_{2^n}$ for some integers $m, m'$, where
$0 \le i', j' \le 2^{n-1}$ and $i' \equiv \pm i \pmod{4}$ and
$j' \equiv \pm j \pmod{4}$. We can consider $V_{i'}$ and $V_{j'}$ as
modules for $H \cong C_{2^{n-1}}$ and the claim follows from induction
and part (b).
\end{proof}

%%%%%%%%%%%%%%%%%%%%%%%%%%%%%%%%%%%%%%%%%%%%%%%%%%%%%%%%%%%%%%%%%%%%%%%%%%%%%%%%%%%%

\section{Main Theorem}
\label{sec:key}

From now on we assume that $k=\F_2$ is a field with 2 elements and $G$
is a cyclic group of order $2^n$. For $0 \leq s \leq 2^{n-1}$, we
know from Lemma~\ref{la:KVWhom} that $K(V_{2^{n-1}+s}, V_s^2)$ is
acyclic and that its homology in degree~0 is
$S(V_{2^{n-1}+s})/(V_s^2)$. It will turn out that
$S(V_{2^{n-1}+s})/(V_s^2)$ is closely related to the exterior
algebra $\Lambda(V_{2^{n-1}+s})$, so it is natural to study the
structure of the graded ring
$\widetilde{S}(V_{2^{n-1}+s}) = \bigoplus_{r \ge 0}
\widetilde{S}^r(V_{2^{n-1}+s}) := S(V_{2^{n-1}+s})/(V_s^2)$ as a
$kG$-module. For a non-negative integer $m$ write
$\widetilde{S}^{<m}(V_{2^{n-1}+s}) = \bigoplus_{r=0}^{m-1}
\widetilde{S}^r(V_{2^{n-1}+s})$ and use a similar notation for other
graded modules.
Let $\widetilde{N}(V_{2^{n-1}+s})$ denote the kernel of the natural
epimorphism $\widetilde{S}(V_{2^{n-1}+s}) \rightarrow \Lambda(V_{2^{n-1}+s})$.

Choose a $k$-basis $\{x_1, x_2, \dots, x_{2^{n-1}+s}\}$ of
$V_{2^{n-1}+s}$ as in Section~\ref{sec:cyc2}. For simplicity, write
$x_{top} := x_{2^{n-1}+s}$, $x_{top-1} := x_{2^{n-1}+s-1}$ and so
on. Each element of $S(V_{2^{n-1}+s})$ can be written uniquely as a
polynomial in $x_1$, $x_2$, \dots, $x_{top}$. Set
$a := \prod_{i=1}^{2^n} (g^i x_{top}) \in S(V_{2^{n-1}+s})$.
If $s < 2^{n-1}$ let $\widetilde{a}$ be the image of $a$ in
$\widetilde{S}(V_{2^{n-1}+s})$, and if $s = 2^{n-1}$ let
$\widetilde{a}$ be the image of the element
$\prod_{t \in G/C_2}(t x_{top}^2)$ in $\widetilde{S}(V_{2^{n}})$.
In the latter case $\widetilde{a}$ is still invariant, because
$g^{2^{n-1}}x_{2^n} = x_{2^n} + x_{2^{n-1}}$ and so $x_{top}^2 \in
\widetilde{S}(V_{2^{n}})$ is invariant under $C_2$. 
In all cases, $a$ is homogeneous of degree~$2^n$, has degree~$2^n$
when considered as a polynomial in $x_{top}$, the elements $a$ and
$\widetilde{a}$ are invariant under the action of $G$,
and the image of $\widetilde{a}$ in $\Lambda(V_{2^{n-1}+s})$ is $0$.
If $s=2^{n-1}$ then we also write $\widetilde{b}$ for the image of the
element $\prod_{i=1}^{2^n} (g^i x_{top})$ in
$\widetilde{S}(V_{2^n})$.

The next theorem is our main result. Since any representation of
$G \cong C_{2^n}$ over a field of characteristic 2 can be written in
$\F_2$, part (d) implies Theorem~\ref{ext}, but we record the other
parts since they are also of interest and they form an integral part
of the proof.

\begin{theorem} \label{thm:main}
Let $n$ and $s$ be integers such that $n \ge 1$ and
$0 \le s \le 2^{n-1}$.
\begin{enumerate}
\item[(a)] (Separation) The complex
$K(V_{2^{n-1}+s}, V_s^2)$ of $kG$-modules is separated.

\item[(b)] (Periodicity) For $s < 2^{n-1}$ we have
  $\widetilde{S}(V_{2^{n-1}+s}) \cong_{\ind} k[\widetilde{a}]
  \otimes \widetilde{S}^{< 2^n}(V_{2^{n-1}+s})$ as graded
  $kG$-modules. For $s=2^{n-1}$ we have $\widetilde{S}(V_{2^n})
  \cong_{\ind} k[\widetilde{a}] \otimes (\widetilde{S}^{< 2^n}(V_{2^n})
  \oplus k \widetilde{b})$. In both cases the isomorphism from right
  to left is induced by the product in $\widetilde{S}(V_{2^{n-1}+s})$.

\item[(c)] (Splitting) The short exact sequence of graded $kG$-modules
\[
0 \longrightarrow \widetilde{N}(V_{2^{n-1}+s}) \longrightarrow
\widetilde{S}(V_{2^{n-1}+s}) \longrightarrow \Lambda(V_{2^{n-1}+s})
\longrightarrow 0
\]
is split and
$\widetilde{N}(V_{2^{n-1}+s}) = \widetilde{a}
\widetilde{S}(V_{2^{n-1}+s}) \oplus \widetilde{I}$, where
$\widetilde{I}$ is a $kG$-module induced from $H$.

\item[(d)] (Exterior powers) For each $r \ge 0$ we have the following
  isomorphism of $kG$-modules
\[
\Lambda^r(V_{2^{n-1}+s}) \cong_{\proj} \bigoplus_{\genfrac{}{}{0pt}{}{i,j \ge 0}{2i+j=r}}
  \Omega_{2^n}^{i+j} (\Lambda^i(V_s) \otimes \Lambda^j(V_{2^{n-1}-s})).
\]
\end{enumerate}
\end{theorem}

The case $s=0$ is a little unnatural, but we need it for the
induction, because the restriction of $V_{2^{n-1}+1}$ is
$V_{2^{n-2}+1} \oplus V_{2^{n-2}}$. 

It is sometimes more succinct to consider Hilbert series with
coefficients in the Green ring (possibly modulo projectives or induced
modules). For more details see \cite{HS}. In particular, we consider
the following series associated to a $kG$-module $V$:
\[
\begin{array}{ll}
\lambda _t(V) = \sum\limits_{r=0}^\infty \Lambda ^r(V)t^r, \quad &
\sigma _t (V) = \sum\limits_{r=0}^\infty S^r(V) t^r, \\
\widetilde{\sigma} _t (V) = \sum\limits_{r=0}^\infty \widetilde{S}^r(V) t^r, &
\lambda _t^{\Omega}(V) = \sum\limits_{r=0}^\infty \Omega ^r \Lambda ^r(V)t^r.
\end{array}
\]
The last of these requires $G$ to be specified in order for the
$\Omega$ to be determined; it is naturally considered modulo projectives.
They all commute with restriction and turn direct sums of modules into
products of series. This is all an easy consequence of the
corresponding properties of the corresponding functors on modules,
except perhaps for $\lambda _t^{\Omega}(V \oplus W)$, where we need
the formula
$\Omega ^r V \otimes \Omega ^s W \cong_{\proj} \Omega ^{r+s}(V \otimes W)$.

Many of our statements about modules imply Hilbert series versions.
\begin{equation} \label{eq:HSform}
\begin{array}{ll}
\hspace*{0.005cm} \sigma _t (V_{2^{n-1}+s}) =_{\ind}
  \lambda^\Omega_t(V_{2^{n-1}-s})(1-t^{2^n})^{-1} & \text{Theorem~\ref{sym}}
\rule[-0.2cm]{0cm}{0.4cm}\\
\rule{0cm}{0.4cm}
\widetilde{\sigma} _t (V_{2^{n-1}+s}) =_{\ind} \lambda
_{t^2}^{\Omega}(V_{s}) \sigma_t(V_{2^{n-1}+s}) & \text{Separation \ref{thm:main}(a)}
\rule[-0.2cm]{0cm}{0.4cm}\\
\rule{0cm}{0.4cm}
\widetilde{\sigma} _t (V_{2^{n-1}+s}) =_{\ind} \lambda
  _t(V_{2^{n-1}+s})(1-t^{2^n})^{-1} & \text{Splitting and periodicity \ref{thm:main}(b),(c)}
\rule[-0.2cm]{0cm}{0.4cm}\\
\rule{0cm}{0.4cm}
\lambda _t(V_{2^{n-1}+s}) =_{\proj} \lambda _{t^2}^{\Omega}(V_{s})
  \lambda _t^{\Omega}(V_{2^{n-1}-s}) & \text{Exterior powers \ref{thm:main}(d),}
\end{array}
\end{equation}
where the symbols $=_{\ind}$ and $=_{\proj}$ mean that we consider the
coefficients only modulo induced or projective direct summands,
respectively. The first and last of the above identities are, in fact,
equivalent to the original versions. The second identity follows from
Theorem~\ref{thm:main} (a), Lemma~\ref{la:homol1} and
Lemma~\ref{la:KVWhom} (once the theorem is proved).

\begin{remark} An easy calculation shows that, for fixed $n$ and $s$,
  the last of the formulas in (\ref{eq:HSform}) follows formally from
  the first three if we are satisfied with only $=_{\ind}$.
\end{remark}

\begin{remark} The proof of Theorem~\ref{sym} in~\cite{scyclic}
  actually gives a more precise formula than the first one in
  (\ref{eq:HSform}). It works by showing that the 
  complex $K(V_{2^n}, V_{2^{n-1}-s})$ defined in~\cite{scyclic} is
  separated and then applying Lemma~\ref{la:homol1}; note that the
  definition of $K(V_{2^n}, V_{2^{n-1}-s})$ in~\cite{scyclic} is
  different from our Definition~\ref{def:koszul}. The result is that
  $\sigma_t (V_{2^{n-1}+s}) = _{\proj} \sigma _t
  (V_{2^{n}})\lambda^{\Omega} _t (V_{2^{n-1}-s})$. Since $V_{2^n}$ can
  be given a basis that is permuted by $G$, each $S^r(V_{2^n})$ has a
  monomial basis that is permuted. For small $n$, the decomposition of
  $\sigma_t(V_{2^{n}})$ can be calculated by hand; in general the
  calculation can be organized using \cite[Proposition~2.2]{scyclic}. Alternatively,
  \cite[Proposition~2.2]{scyclic} can be applied directly to $\sigma _t(V_{2^{n-1}+s})$.
\end{remark}

The next six sections are devoted to the proof of Theorem~\ref{thm:main}
by induction on $n$.

%%%%%%%%%%%%%%%%%%%%%%%%%%%%%%%%%%%%%%%%%%%%%%%%%%%%%%%%%%%%%%%%%%%%%%%%%%%%%%%%

\section{The Case $n=1$}
\label{sec:n=1}

In this section we start the inductive proof of
Theorem~\ref{thm:main}. Suppose that $n=1$, so we have to prove the
statements in Theorem~\ref{thm:main} for $s \in \{0,1\}$. With these
assumptions on $n$ and $s$, parts (b)-(d) of Theorem~\ref{thm:main} can
easily be verified by a direct calculation. In fact, in~(b) one obtains
isomorphisms of $kG$-modules (not only modulo induced summands), and
in~(c) one gets $\widetilde{I} = 0$. Separation for $s=0$ is trivial.

Let us consider part (a) for $n=s=1$. We have to show that for each
$r > 0$ the short exact sequence
$0 \rightarrow S^{r-2}(V_2) \rightarrow S^r(V_2) \rightarrow
\widetilde{S}^r(V_2) \rightarrow 0$
of $kG$-modules is separated at $S^r(V_2)$. If $r$ is odd, then
$S^r(V_2)$ is induced by Theorem~\ref{sym}, hence projective, and so
separation is obviously true. Separation is trivial for $r=0$.
For even $r > 0$, a direct calculation and Theorem~\ref{sym} show
that $\widetilde{S}^r(V_2) \cong V_1 \oplus V_1 \cong_{\ind} S^r(V_2)
\oplus \Omega_2^{-1}S^{r-2}(V_2)$, and so separation follows from
Lemma~\ref{la:looks_sep}.

Sections~\ref{sec:per}-\ref{sec:ext} comprise the inductive step in the proof
of Theorem~\ref{thm:main}. In these sections we always assume that
$n > 1$ is an integer and that Theorem~\ref{thm:main} holds for all
smaller values of~$n$. Throughout these sections the notation remains
the same as in Sections~\ref{sec:cyc2} and \ref{sec:key}; thus  
$G = \langle g \rangle \cong C_{2^n}$ is a cyclic group of order
$2^n$, $k = \F_2$ is a field with two elements and $s$ is an integer
such that $0 \le s \le 2^{n-1}$.

%%%%%%%%%%%%%%%%%%%%%%%%%%%%%%%%%%%%%%%%%%%%%%%%%%%%%%%%%%%%%%%%%%%%%%%%%%%%%%%%%%

\section{Periodicity}
\label{sec:per}

In this section we
prove part (b) of Theorem~\ref{thm:main}, assuming that parts (a)-(d)
of the theorem hold for all smaller values of $n$.

Let $H$ be the unique maximal subgroup of $G$ and let
$\{x_1, x_2, \dots, x_{top}\}$ be a $k$-basis of $V_{2^{n-1}+s}$
as in Section~\ref{sec:key}. We choose $G$-invariant elements
$a \in S^{2^n}(V_{2^{n-1}+s})$ and
$\widetilde{a} \in \widetilde{S}^{2^n}(V_{2^{n-1}+s})$ as in
Section~\ref{sec:key}. Let $T(V_{2^{n-1}+s})$ be the $kG$-submodule of
$S(V_{2^{n-1}+s})$ spanned by the monomials in $x_1$, \dots, $x_{top}$
that are not divisible by $x_{top}^{2^n}$. We have
$S(V_{2^{n-1}+s}) \cong k[a] \otimes T(V_{2^{n-1}+s})$ as
$kG$-modules; see~\cite[Lemma~1.1]{scyclic}. So
$T^{<2^n}(V_{2^{n-1}+s}) = S^{<2^n}(V_{2^{n-1}+s})$. Notice that the
periodicity of $S(V_{2^{n-1}+s})$ in \cite[Theorem~1.2]{scyclic} is
equivalent to $T^{\geq 2^n}(V_{2^{n-1}+s})$ being induced. In fact,
we know something stronger from \cite[Corollary~3.11]{scyclic}, namely
that $T^{>2^{n-1}-s}(V_{2^{n-1}+s})$ is induced.

We can make the same construction for $\widetilde{S}(V_{2^{n-1}+s})$,
obtaining $\widetilde{S}(V_{2^{n-1}+s}) \hspace{-0.05cm}
\cong \hspace{-0.05cm} k[\widetilde{a}] \otimes
\widetilde{T}(V_{2^{n-1}+s})$ as $kG$-modules.

Define $L(V_{2^{n-1}+s}, V_s^2)$ to be the subcomplex of
$K(V_{2^{n-1}+s}, V_s^2)$ defined using $T(V_{2^{n-1}+s})$
instead of $S(V_{2^{n-1}+s})$, that is
\[
\dots \stackrel{d}{\longrightarrow} T(V_{2^{n-1}+s}) \otimes \Lambda^2(V_s)
\stackrel{d}{\longrightarrow} T(V_{2^{n-1}+s}) \otimes V_s
\stackrel{d}{\longrightarrow} T(V_{2^{n-1}+s}),
\]
where the boundary morphisms are as in Definition~\ref{def:koszul}
(this can be done since the $x_{top}$ used in the definition of
$T(V_{2^{n-1}+s})$ is not contained in $V_s$). Thus $L(V_{2^{n-1}+s},
V_s^2)$ is a complex of graded $kG$-modules; it is exact except in
degree 0, where the homology is $H_0(L(V_{2^{n-1}+s}, V_s^2))$, which
is isomorphic to  $\widetilde{T}(V_{2^{n-1}+s})$ as 
a $kG$-module. Notice that, by construction, the complexes
$K(V_{2^{n-1}+s}, V_s^2)$ and $k[a] \otimes L(V_{2^{n-1}+s}, V_s^2)$
of $kG$-modules are isomorphic. In particular, note for later use that
one of them is separated (over $G$ or over $H$) if and only if the
other is so too. 

From now on we fix $s$ and abbreviate the notation to just $S$, $T$,
$K$, $L$, etc.

Suppose that $s < 2^{n-1}$. We claim that $L^r_i = T^{r-2i} \otimes \Lambda^i(V_s)$ is induced for all
$i \ge 0$, $r \ge 2^n$. We may assume that $i \leq s$. Then
$r-2i \ge 2^n-2s > 2^{n-1}-s$ and so $T^{r-2i}$ is induced. Thus $L^r$
is a complex of induced $kG$-modules for each $r \ge 2^n$.

Consider the restriction of the complex $K$ to the subgroup $H$.
It decomposes as a tensor product of two complexes, by
Lemma~\ref{la:dirsum}. Each of these is separated, by induction and
Theorem~\ref{thm:main} (a), hence so is their product, by
Lemma~\ref{la:tensep}. It follows that for each $r \ge 0$, the complex
$L^r$ is separated on restriction to $H$. We have just seen that $L^r$
is a complex of induced modules for all $r \ge 2^n$. Thus, for each
$r \ge 2^n$, the complex $L^r$is separated, by
Proposition~\ref{prop:homol2}. Now Lemma~\ref{la:homol1} shows
that $H_0(L^{\geq 2^n})$ is induced. But this is exactly
$\widetilde{T}^{\geq 2^n}$, so $\widetilde{S}(V_{2^{n-1}+s}) \cong
k[\widetilde{a}] \otimes \widetilde{T} \cong_{\ind} k[\widetilde{a}]
\otimes \widetilde{T}^{<2^n} = k[\widetilde{a}] \otimes
\widetilde{S}^{<2^n}$ is periodic if $s < 2^{n-1}$.

Now suppose that $s = 2^{n-1}$. By the same argument as for
$s < 2^{n-1}$, we see that $\widetilde{T}^{>2^n}$ is induced. To
complete the proof of Theorem~\ref{thm:main} (b) we have to show
that $\widetilde{S}^{2^n} \cong_{\ind} k\widetilde{a} \oplus k
\widetilde{b}$. Set $y_i := g^{2^n-i}x_{2^n}$ for
$i=1,2,\dots,2^n$, so $\{y_1,\dots,y_{2^n}\}$ is a $k$-basis of
$V_{2^n}$ which is permuted by $G$. A basis for $V_{2^{n-1}} <
V_{2^n}$ is given by $g^{2^{n-1}}y_i - y_i=y_{i+2^{n-1}}-y_i$ for
$i=1,2,\dots,2^{n-1}$.  Write $\widetilde{y}_i$ 
for the image of $y_i$ in $\widetilde{S}(V_{2^n})$, so
$\widetilde{y}_i^2 = \widetilde{y}_{i+2^{n-1}}^2$ for
$i=1,\dots,2^{n-1}$. The set consisting of all monomials of degree
$2^n$ in all the $\widetilde{y}_i$ such that $\widetilde{y}_1$, \dots,
$\widetilde{y}_{2^{n-1}}$ only occur to the power at most 1 forms a $k$-basis
for $\widetilde{S}(V_{2^n})$. The group $G$ permutes these monomials
and it is straightforward to check that there are two
invariant monomials, namely
$\widetilde{y}_1 \widetilde{y}_2 \cdots \widetilde{y}_{2^n} =
\widetilde{b}$ and $\widetilde{y}_{2^{n-1}+1}^2
\widetilde{y}_{2^{n-1}+2}^2 \cdots \widetilde{y}_{2^n}^2 =
\widetilde{a}$; the rest span induced submodules. This completes the
proof of periodicity.

%%%%%%%%%%%%%%%%%%%%%%%%%%%%%%%%%%%%%%%%%%%%%%%%%%%%%%%%%%%%%%%%%%%%%%%%%%%%%%%%%%%%%%%%%%%%%%%%

\section{Splitting}
\label{sec:split_sm}

 In this section we
prove part (c) of Theorem~\ref{thm:main}, assuming the whole of the
theorem for smaller~$n$.

Let $H$ be the unique maximal subgroup of $G$ and
$\{x_1, x_2, \dots, x_{top}\}$ a $k$-basis of $V_{2^{n-1}+s}$ as in
Section~\ref{sec:key}. As in Theorem~\ref{thm:main} we write
$\widetilde{N}(V_{2^{n-1}+s})$ for the kernel of the natural surjection
$\widetilde{S}(V_{2^{n-1}+s}) \stackrel{f}{\rightarrow} \Lambda(V_{2^{n-1}+s})$.
The following proposition deals with the structure of
$\widetilde{S}(V_{2^{n-1}+s})$ in degrees less than $2^n$.

\begin{proposition} \label{prop:split_sm}
For any integer $s$ such that $0 \le s \le 2^{n-1}$, the short exact sequence
\begin{equation} \label{eq:splitseq_sm}
0 \longrightarrow \widetilde{N}^{<2^n}(V_{2^{n-1}+s}) \longrightarrow
\widetilde{S}^{<2^n}(V_{2^{n-1}+s}) \stackrel{f}{\longrightarrow} \Lambda^{<2^n}(V_{2^{n-1}+s})
\longrightarrow 0
\end{equation}
of graded $kG$-modules is split, and
$\widetilde{N}^{<2^n}(V_{2^{n-1}+s})$ is induced from $H$.
\end{proposition}

Before starting with the proof of Proposition~\ref{prop:split_sm} we
introduce some further notation. As described at the beginning of
Section~\ref{sec:cyc2}, we have
$V_{2^{n-1}+s} \hspace{-0.1cm} \downarrow^G_H = V_{2^{n-2}+s'} \oplus
V_{2^{n-2}+s''}$ where $0 \le s', s'' \le 2^{n-2}$ and $s'=s''$ or
$s' = s''+1$. The $kH$-submodule $V_{2^{n-2}+s'}$ of $V_{2^{n-1}+s}$
has the $k$-basis $\{x_{top}, x_{top-2}, x_{top-4}, \dots\}$ and the
$kH$-submodule $V_{2^{n-2}+s''}$ has the $k$-basis
$\{x_{top-1}, x_{top-3}, x_{top-5}, \dots\}$. We write
$\widetilde{S}'(V_{2^{n-2}+s'})$ and
$\widetilde{S}''(V_{2^{n-2}+s''})$ for $S(V_{2^{n-2}+s'})/(V_{s'}^2)$
and $S(V_{2^{n-2}+s''})/(V_{s''}^2)$, respectively. So the $x_i$ with
odd $i$ and the $x_i$ with even $i$ provide natural
embeddings $\widetilde{S}'(V_{2^{n-2}+s'}) \rightarrow
\widetilde{S}(V_{2^{n-1}+s})$ and $\widetilde{S}''(V_{2^{n-2}+s''})
\rightarrow \widetilde{S}(V_{2^{n-1}+s})$ of $kH$-modules, and we
have
\[
\widetilde{S}(V_{2^{n-1}+s}) \cong \widetilde{S}'(V_{2^{n-2}+s'})
\otimes \widetilde{S}''(V_{2^{n-2}+s''})
\]
as $kH$-modules, where the isomorphism is given by
$f_1 \otimes f_2 \mapsto f_1 \cdot f_2$.

Choose $a' \in S(V_{2^{n-2}+s'})$,
$\widetilde{a}' \in \widetilde{S}(V_{2^{n-2}+s'})$ according to the
description preceding Theorem~\ref{thm:main}, but working
over~$H$. Thus $a'$ is homogeneous of degree~$2^{n-1}$, has
degree~$2^{n-1}$ when considered as a polynomial in
$x_{top}$. Furthermore, $a'$ and $\widetilde{a}'$ are invariant under
the action of~$H$ and $\widetilde{a}'$ has image 0 in $\Lambda(V_{2^{n-2}+s'})$.
Similarly, choose $a'' \in S(V_{2^{n-2}+s''})$ and $\widetilde{a}'' \in
\widetilde{S}''(V_{2^{n-2}+s''})$. So $a''$ is homogeneous
of degree $2^{n-1}$, has degree~$2^{n-1}$ when considered
as a polynomial in $x_{top-1}$ and $a''$ and $\widetilde{a}''$ are
invariant under the action of~$H$ and $\widetilde{a}''$ has image 0 in
$\Lambda(V_{2^{n-2}+s''})$.

By induction, $\widetilde{a}'$ and $\widetilde{a}''$ are periodicity
generators of $\widetilde{S}'(V_{2^{n-2}+s'})$ and
$\widetilde{S}''(V_{2^{n-2}+s''})$, respectively. That is, we have
\[
\widetilde{S}'(V_{2^{n-2}+s'}) \cong_{\ind} k[\widetilde{a}'] \otimes
\widetilde{S}'^{<2^{n-1}}(V_{2^{n-2}+s''}) \quad \text{and} \quad
\widetilde{S}''(V_{2^{n-2}+s''}) \cong_{\ind} k[\widetilde{a}''] \otimes
\widetilde{S}''^{<2^{n-1}}(V_{2^{n-2}+s''})
\]
or the variant with $\widetilde{b}'$ or $\widetilde{b}''$ if
$s'=2^{n-2}$ or $s''=2^{n-2}$.

\begin{lemma} \label{la:stsmalldeg}
Let $s$ be an integer such that $0 \le s \le 2^{n-1}$ and
let $\widetilde{a}'$ be a periodicity generator for
$\widetilde{S}'(V_{2^{n-2}+s'})$ as above. Then
\[
{\widetilde{S}^{<2^n}(V_{2^{n-1}+s}) \hspace{-0.1cm} \downarrow^G_H} \, = \,
\widetilde{S}^{<2^{n-1}}(V_{2^{n-1}+s}) \oplus \widetilde{a}'
\widetilde{S}^{<2^{n-1}}(V_{2^{n-1}+s}) \oplus (g\widetilde{a}')
\widetilde{S}^{<2^{n-1}}(V_{2^{n-1}+s}) \oplus \widetilde{X}
\]
as $kH$-modules, where the $kH$-submodule $\widetilde{X}$ is generated
as a $k$-vector space by the images of all monomials
$x \in \bigoplus_{r=2^{n-1}}^{2^n-1} S^r(V_{2^{n-1}+s})$ such that $x$
has degree strictly less than~$2^{n-1}$ when considered as
a polynomial in $x_{top}$ and $x$ has degree strictly less
than~$2^{n-1}$ when considered as a polynomial in $x_{top-1}$.
\end{lemma}

\begin{proof}
We give all monomials in $S(V_{2^{n-1}+s})$ the lexicographic order with
$x_{top-1} > x_{top} > x_{top-2} > \dots > x_1$. Let
$h \in \bigoplus_{r=2^{n-1}}^{2^n-1} S^r(V_{2^{n-1}+s})$. Since
$g\widetilde{a}'$ has leading term $x_{top-1}^{2^{n-1}}$ we can write 
$h$ as $h = h_1 \cdot g\widetilde{a}' + h_2$ where 
$h_1 \in \widetilde{S}^{<2^{n-1}}(V_{2^{n-1}+s})$ and 
$h_2 \in \widetilde{S}^{<2^n}(V_{2^{n-1}+s})$ such that $h_2$ has
degree $<2^{n-1}$ when considered as a polynomial in~$x_{top-1}$.
Then, because $\widetilde{a}'$ has leading term $x_{top}^{2^{n-1}}$
and only involves monomials in $x_{top}$, $x_{top-2}$, $x_{top-4}$,
\dots, we can find $h_3 \in \widetilde{S}^{<2^{n-1}}(V_{2^{n-1}+s})$
and $h_4 \in \widetilde{X}$ such that 
$h_2 = h_3 \cdot \widetilde{a}' + h_4$. Thus,
$\bigoplus_{r=2^{n-1}}^{2^n-1} S^r(V_{2^{n-1}+s})$ is the sum of
$\widetilde{a}'\widetilde{S}^{<2^{n-1}}(V_{2^{n-1}+s})$, 
$(g\widetilde{a}')\widetilde{S}^{<2^{n-1}}(V_{2^{n-1}+s})$ and 
$\widetilde{X}$. Comparing dimensions, we see that this sum has to be
direct and Lemma~\ref{la:stsmalldeg} follows.
\end{proof}

We are now ready to prove Proposition~\ref{prop:split_sm}.

\begin{proof} (of Proposition~\ref{prop:split_sm})
We study the restriction of the sequence (\ref{eq:splitseq_sm}) to the
maximal subgroup~$H$. By Lemma~\ref{la:stsmalldeg}, the middle term is
\[
{\widetilde{S}^{<2^n}(V_{2^{n-1}+s}) \hspace{-0.1cm} \downarrow^G_H} =
  \widetilde{S}^{<2^{n-1}}(V_{2^{n-1}+s}) \oplus \widetilde{J} \oplus \widetilde{X},
\]
where $\widetilde{J} := \widetilde{a}' \widetilde{S}^{<2^{n-1}}(V_{2^{n-1}+s})
\oplus (g\widetilde{a}') \widetilde{S}^{<2^{n-1}}(V_{2^{n-1}+s})$.
Owing to the choice of $\widetilde{a}'$, we have
$\widetilde{a}' \in \widetilde{N}^{<2^n}(V_{2^{n-1}+s})$ and therefore
$\widetilde{J} \subseteq \widetilde{N}^{<2^n}(V_{2^{n-1}+s})$.
In fact, by construction, $\widetilde{J}$ is a $kG$-submodule of
$\bigoplus_{r=2^{n-1}}^{2^n-1}\widetilde{N}^r(V_{2^{n-1}+s})$ and is
induced from $H$. We consider the exact sequence of $kG$-modules
\begin{equation} \label{eq:splitJ}
0 \longrightarrow \widetilde{J} \longrightarrow
\bigoplus_{r=2^{n-1}}^{2^n-1}\widetilde{S}^r(V_{2^{n-1}+s})
\longrightarrow \overline{X} \longrightarrow 0,
\end{equation}
where $\overline{X} :=
\bigoplus_{r=2^{n-1}}^{2^n-1}\widetilde{S}^r(V_{2^{n-1}+s}) /
\widetilde{J}$. We know from Lemma~\ref{la:stsmalldeg} that the
sequence~(\ref{eq:splitJ}) is split when restricted to $H$ and
${\overline{X} \hspace{-0.1cm} \downarrow^G_H} \cong  \widetilde{X}$
as $kH$-modules. Since $\widetilde{J}$ is induced from $H$ it is
relatively $H$-injective, and so the sequence (\ref{eq:splitJ}) splits
over $kG$ (see~\cite[Theorem (19.2)]{cr}). Thus $\widetilde{J}$ is a
direct summand of $\widetilde{S}^{<2^n}(V_{2^{n-1}+s})$ over $kG$, so
there is a $kG$-submodule $\widetilde{J}''$ of
$\widetilde{S}^{<2^n}(V_{2^{n-1}+s})$ such that
$\widetilde{S}^{<2^n}(V_{2^{n-1}+s}) = \widetilde{J} \oplus
\widetilde{J}''$. Since $\widetilde{J} \subseteq
\widetilde{N}^{<2^n}(V_{2^{n-1}+s})$, it follows that
$\widetilde{N}^{<2^n}(V_{2^{n-1}+s}) = \widetilde{J} \oplus
\widetilde{J}'$, where $\widetilde{J}' := \widetilde{J}'' \cap
\widetilde{N}^{<2^n}(V_{2^{n-1}+s})$. We have
\[
{\widetilde{J} \hspace{-0.1cm} \downarrow^G_H} = \widetilde{a}'
{\widetilde{S}^{<2^{n-1}}(V_{2^{n-1}+s}) \hspace{-0.1cm} \downarrow^G_H}
\oplus (g\widetilde{a}') {\widetilde{S}^{<2^{n-1}}(V_{2^{n-1}+s}) \hspace{-0.1cm}
\downarrow^G_H} .
\]
Because ${\widetilde{S}^{<2^{n-1}}(V_{2^{n-1}+s}) \hspace{-0.1cm}
\downarrow^G_H} \cong (\widetilde{S}'(V_{2^{n-2}+s'}) \otimes
\widetilde{S}''(V_{2^{n-2}+s''}))^{<2^{n-1}}$, we have, by induction
and ignoring the grading,
\begin{equation} \label{eq:resJ}
{\widetilde{J} \hspace{-0.1cm} \downarrow^G_H} \cong_{\ind}
(\Lambda' \otimes \Lambda'')^{<2^{n-1}} \oplus (\Lambda' \otimes
\Lambda'')^{<2^{n-1}}.
\end{equation}
Here we write $\Lambda' \otimes \Lambda''$ for the graded $kH$-module
$\Lambda(V_{2^{n-2}+s'}) \otimes \Lambda(V_{2^{n-2}+s''})$. Restricting
the sequence (\ref{eq:splitseq_sm}) to $H$, we obtain the sequence
\[
0 \rightarrow {\widetilde{N}^{<2^n}(V_{2^{n-1}+s}) \hspace{-0.1cm}
  \downarrow^G_H} \rightarrow (\widetilde{S}'(V_{2^{n-2}+s'}) \otimes
\widetilde{S}''(V_{2^{n-2}+s''}))^{<2^n} \rightarrow (\Lambda'
\otimes \Lambda'')^{<2^n} \rightarrow 0
\]
of $kH$-modules, which is split by induction. Thus, by induction
again, we obtain
\begin{equation} \label{eq:resK}
{\widetilde{N}^{<2^n}(V_{2^{n-1}+s}) \hspace{-0.1cm} \downarrow^G_H}
\cong_{\ind} \widetilde{a}'(\Lambda' \otimes \Lambda'')^{<2^{n-1}} \oplus
\widetilde{a}''(\Lambda' \otimes \Lambda'')^{<2^{n-1}}.
\end{equation}
Equations (\ref{eq:resJ}) and (\ref{eq:resK}) imply that
${\widetilde{J} \hspace{-0.1cm} \downarrow^G_H} \oplus
{\widetilde{J}' \hspace{-0.1cm} \downarrow^G_H} \cong
({\widetilde{J} \oplus \widetilde{J}') \hspace{-0.1cm} \downarrow^G_H} \cong
{\widetilde{N}^{<2^n}(V_{2^{n-1}+s}) \hspace{-0.1cm} \downarrow^G_H} \cong_{\ind}
{\widetilde{J} \hspace{-0.1cm} \downarrow^G_H}$. It follows that
${\widetilde{J}' \hspace{-0.1cm} \downarrow^G_H}$ is induced from
proper subgroups of $H$. By Lemma~\ref{la:resind} the $kG$-module
$\widetilde{J}'$ is induced from $H$, and hence
$\widetilde{N}^{<2^n}(V_{2^{n-1}+s}) = \widetilde{J} \oplus
\widetilde{J}'$ is induced from $H$. We have just seen that sequence
(\ref{eq:splitseq_sm}) is split on restriction to $H$;
since $\widetilde{N}^{<2^n}(V_{2^{n-1}+s})$ is relatively
$H$-injective the sequence must split over $kG$. This completes the
proof of Proposition~\ref{prop:split_sm}.
\end{proof}

The following corollary provides a connection between
$\widetilde{S}(V_{2^{n-1}+s})$ and the exterior powers of
$V_{2^{n-1}+s}$ in degrees less than $2^n$.

\begin{corollary} \label{cor:structSt_sm}
For $r$ and $s$ integers such that
$0 \le s \le 2^{n-1}$ and $0 \le r < 2^n$, the map $f$ induces an
isomorphism of $kG$-modules modulo induced summands
\[
\widetilde{S}^r(V_{2^{n-1}+s}) \cong_{\ind} \Lambda^r(V_{2^{n-1}+s}).
\]
\end{corollary}

\begin{proof}
This is clear from Proposition~\ref{prop:split_sm} (for $n > 1$) and
Section~\ref{sec:n=1} (for $n=1$).
\end{proof}

We can now prove Theorem~\ref{thm:main} (c). For $s<2^{n-1}$
we have $\widetilde{S}^{<2^n}(V_{2^{n-1}+s}) \cong
\Lambda(V_{2^{n-1}+s}) \oplus X$, where $X$ is induced, so part (c) of
Theorem~\ref{thm:main} follows from part (b). For $s=2^{n-1}$ we have
$\widetilde{S}^{<2^n}(V_{2^{n-1}+s}) \oplus k\widetilde{b} \cong
\Lambda(V_{2^{n-1}+s}) \oplus X'$, where $X'$ is induced. Note that
$\widetilde{b}$ maps to a generator of $\Lambda^{2^n}(V_{2^{n-1}+s})$.
Again, part (c) of Theorem~\ref{thm:main} is a consequence of (b).

%%%%%%%%%%%%%%%%%%%%%%%%%%%%%%%%%%%%%%%%%%%%%%%%%%%%%%%%%%%%%%%%%%%%%%%%%%%%%%%%

\section{Preparation for Separation}
\label{sec:sep1}

In this section we prepare for
the proof of part (a) of Theorem~\ref{thm:main}, assuming the whole of
the theorem for smaller $n$.

Let $H$ be the unique maximal subgroup of $G$ and
let $\{x_1, x_2, \dots, x_{top}\}$ be a $k$-basis of $V_{2^{n-1}+s}$ as in
Section~\ref{sec:key}. The main goal of this section is to develop
useful criteria for the complex $K(V_{2^{n-1}+s}, V_s^2)$ to be
separated.

\begin{lemma} \label{la:sep_in_posdeg}
Let $r$, $s$ be non-negative integers such that
$0 \le s \le 2^{n-1}$. Suppose that for each $0 \le r' < r$ with
$r' \equiv r$ mod $2$, the complex $K^{r'}(V_{2^{n-1}+s}, V_s^2)$ is
separated at $K_0^{r'}(V_{2^{n-1}+s}, V_s^2) = S^{r'}(V_{2^{n-1}+s})$.
Then $K^r(V_{2^{n-1}+s}, V_s^2)$ is separated at
$K_i^r(V_{2^{n-1}+s}, V_s^2) = S^{r-2i}(V_{2^{n-1}+s}) \otimes \Lambda^i(V_s)$
for all $i \ge 1$.
The same is true when $K$ and $S$ are replaced by $L$ and $T$ from
Section~\ref{sec:per}. 
\end{lemma}

\begin{proof}
We only demonstrate the proof for $K$ and $S$; the proof for $L$ and
$T$ is analogous.

We write $V := V_{2^{n-1}+s}$ and $W := V_s$ for short. Fix $i \ge 1$
and consider the boundary morphism
$d_{i+1}: S^{r-2i-2}(V) \otimes \Lambda^{i+1}(W) \rightarrow
S^{r-2i}(V) \otimes \Lambda^i(W)$ in $K^r(V, W^2)$. We have to show
that $\Ima(d_{i+1}) \rightarrow S^{r-2i}(V) \otimes
\Lambda^i(W)$ factors through a projective $kG$-module.
 Since $K^{r-2i}(V, W^2)$ is separated
at $K_0^{r-2i}(V, W^2)$ the inclusion
$(W^2)^{r-2i} \rightarrow S^{r-2i}(V)$ factors through a projective
$kG$-module~$P^{r-2i}$. We can write the inclusion
$\Ima(d_{i+1}) \rightarrow S^{r-2i}(V) \otimes \Lambda^i(W)$ as a
composition of inclusions
\[
\Ima(d_{i+1}) \rightarrow (W^2)^{r-2i} \otimes \Lambda^i(W)
\rightarrow S^{r-2i}(V) \otimes \Lambda^i(W),
\]
where the last map factors through the projective $kG$-module
$P^{r-2i} \otimes \Lambda^i(W)$.
\end{proof}

\begin{lemma}
\label{la:sep_equiv}
Let $s$ and $r$ be integers such that
$0 \le s \le 2^{n-1}$ and $0 < r < 2^n$, and suppose that
the complex $K^i(V_{2^{n-1}+s}, V_s^2)$ is separated for all
$0 \le i < r$. Then the following statements are equivalent:
\begin{enumerate}
\item[(a)] $K^r(V_{2^{n-1}+s}, V_s^2)$ is separated,

\item[(b)] The natural map
$S^r(V_{2^{n-1}+s}) \stackrel{g}{\longrightarrow} \Lambda ^r (V_{2^{n-1}+s})$
is split injective modulo induced summands,

\item[(c)] $\Lambda^r(V_{2^{n-1}+s}) \cong _{\ind}
\bigoplus_{\genfrac{}{}{0pt}{}{i,j \ge 0}{2i+j=r}} \Omega_{2^n}^{i+j}
(\Lambda^i(V_s) \otimes \Lambda^j(V_{2^{n-1}-s}))$.
\end{enumerate}
\end{lemma}

\begin{proof}
We write $S^r := S^r(V_{2^{n-1}+s})$,
$\widetilde{S}^r := \widetilde{S}^r(V_{2^{n-1}+s})$ and
$K^i := K^i(V_{2^{n-1}+s}, V_s^2)$. The conditions on $K^i$ and
Lemma~\ref{la:sep_in_posdeg} show that $K^r$ is separated except,
perhaps, at $K_0^r=S^r$. The restriction of the complex $K^r$ to $H$
decomposes as a tensor product of two complexes, by
Lemma~\ref{la:dirsum}. Each of these is separated by our continuing
induction hypothesis, hence so is their product, by
Lemma~\ref{la:tensep}, and so $K^r$ is separated on restriction to $H$.
Thus the short exact sequence
\begin{equation} \label{eq:seqij}
0 \longrightarrow \Ima (d_1) \stackrel{i}{\longrightarrow} S^r
\stackrel{j}{\longrightarrow}  \widetilde{S}^r \longrightarrow 0
\end{equation}
from $K^r$ is separated at $S^r$ on restriction to $H$ (the
maps $i$, $j$ should not be confused with the indices in part
(c) of the lemma). The separation of $K^r$ in positive
(complex-) degrees and Lemma~\ref{la:homol1} yield the formula
$\Ima (d_1) \cong _{\proj} \bigoplus_{\genfrac{}{}{0pt}{}{i \ge 1,\; j \ge 0}{2i+j=r}}
  \Omega_{2^n}^{i-1} (\Lambda^i(V_s)) \otimes S^j(V_{2^{n-1}+s})$.
Theorem~\ref{sym} now shows that
\begin{equation} \label{eq:imd1}
\Ima (d_1) \cong _{\ind} \bigoplus_{\genfrac{}{}{0pt}{}{i \ge 1,\; j \ge 0}{2i+j=r}}
  \Omega_{2^n}^{i+j-1} (\Lambda^i(V_s) \otimes
  \Lambda^j(V_{2^{n-1}-s})).
\end{equation}

\noindent (a) $\Rightarrow$ (b) Let
$f: \widetilde{S}^r \rightarrow \Lambda^r$ be the natural
surjection, so $g = f \circ j$. By
Proposition~\ref{prop:split_sm}, the map $f$ is split injective modulo induced
summands, and, by Lemma~\ref{la:prop_split}, it is enough to show that
$j$ is split injective modulo induced summands.
By assumption, $S^r = X \oplus M$ for some submodules $X$ and $M$ of $S^r$
such that $X$ is projective  and $\ker(j) = \Ima(d_1) \subseteq X$.
Let $j': S^r \rightarrow X$ be the projection onto $X$ and
$u': X \rightarrow S^r$ the natural embedding. Define
$u: \widetilde{S}^r=j(X) \oplus j(M) \rightarrow S^r, j(x) + j(m)
\mapsto m$ (note that the restriction of $j$ to $M$ is injective). Then
$u \circ j + u' \circ j' = \Id_{S^r}$ and so $j$ is split injective modulo
induced summands.

\medskip

\noindent (b) $\Rightarrow$ (c) Assume (b). The factorization
$g = f \circ j$ and Lemma~\ref{la:prop_split} (b) imply that $j$
is also split injective modulo induced summands. Write $S^r = A' \oplus A''$, where
$A'$ has only non-induced summands and $A''$ is induced. By
Lemma~\ref{la:split_mod_ind}, the restriction of $j$ to
$A'$ is split injective, so $j$ maps~$A'$ injectively into~$\widetilde{S}^r$ and $j(A')$ is
a direct summand of $\widetilde{S}^r$. Factoring out $A'$ and $j(A')$
in (\ref{eq:seqij}) we obtain the short exact sequence
$0 \longrightarrow \Ima (d_1) \stackrel{\overline{i}}{\longrightarrow}
S^r/A' \stackrel{\overline{j}}{\longrightarrow} \widetilde{S}^r/j(A')
\longrightarrow 0$.

As we have seen at the beginning of the proof, $i$ factors
through a projective on restriction to~$H$, and so the same is true
for $\overline{i}$. Thus the complex $\Ima(d_1) \rightarrow S^r/A'$ is
separated on restriction to $H$. Because $S^r/A' \cong A''$ is
induced, the complex is separated, by
Proposition~\ref{prop:homol2}. Lemma~\ref{la:homol1} yields
$\widetilde{S}^r/j(A') \cong_{\proj} S^r/A' \oplus \Omega_{2^n} \Ima(d_1)
\cong_{\ind} \Omega_{2^n} \Ima(d_1)$. Using (\ref{eq:imd1}) we obtain
\begin{equation} \label{eq:modjA'}
\widetilde{S}^r/j(A') \cong_{\ind} \bigoplus_{\genfrac{}{}{0pt}{}{i
\ge 1,\; j \ge 0}{2i+j=r}} \Omega_{2^n}^{i+j} (\Lambda^i(V_s) \otimes
\Lambda^j(V_{2^{n-1}-s})).
\end{equation}
Theorem~\ref{sym} implies $j(A') \cong A' \cong _{\ind} S^r
\cong_{\ind} \Omega_{2^n}^r \Lambda ^r(V_{2^{n-1}-s})$.
Adding the summand $j(A')$ to both sides of (\ref{eq:modjA'}) and
using Corollary~\ref{cor:structSt_sm} gives us the formula in (c).

\medskip

\noindent (c) $\Rightarrow$ (a) Assume that (c) holds. From
Corollary~\ref{cor:structSt_sm}, Theorem~\ref{sym} and (\ref{eq:imd1}) we get
$\widetilde{S}^r \cong_{\ind} S^r \oplus \Omega_{2^n}^{-1}
\Ima(d_1)$. Separation of $K^r$ now follows from applying
Lemma~\ref{la:looks_sep} to the short exact sequence (\ref{eq:seqij}).
\end{proof}

Separation of $K^r(V_{2^{n-1}+s}, V_s^2)$ for $r=0, 1$ is
trivial. We will now prove it for $r = 2$.
Notice that if a non-zero map $V_a \rightarrow V_b$ of $kG$-modules is
to factor through a projective over $C_{2^n}$, then we must have
$a+b>2^n$. This is because the map must factor through the projective
cover $V_{2^n} \twoheadrightarrow V_b$, which has kernel $V_{2^n-b}$,
into which $V_a$ will certainly be mapped if $a \leq 2^n -b$.

\begin{lemma}
\label{la:lambda2}
For any integer $s$ such that
$0 \le s \le 2^{n-1}$ the complex $K^2(V_{2^{n-1}+s}, V_s^2)$ is
separated.
\end{lemma}

\begin{proof}
The complex in question is $V_s \hookrightarrow S^2(V_{2^{n-1}+s})$.
By induction, the map factors through a projective on restriction to
$H$. Write $S^2(V_{2^{n-1}+s}) = A' \oplus A''$, where $A'$ has only
non-induced summands and $A''$ has only induced summands.
The component $V_s \rightarrow A''$ factors through a projective, by
Proposition~\ref{prop:homol2}.

We claim that the component $V_s \rightarrow A'$ must be 0. From
Theorem~\ref{sym}, we know that
$S^2(V_{2^{n-1}+s}) =_{\ind} \Lambda^2(V_{2^{n-1}-s})$; but
$V_{2^{n-1}-s}$ is a module for $C_{2^{n-1}}$, and it follows
that~$A'$ contains only summands of dimension $\le 2^{n-1}$. Let $V_t$
be such a summand, so $t \le 2^{n-1}$ and suppose that there is a
non-zero component $V_s \rightarrow V_t$. It must factor through a
projective on restriction, where it is a map
$V_{s'} \oplus V_{s''} \rightarrow V_{t'} \oplus V_{t''}$, with
$s',s'',t',t'' \leq 2^{n-2}$. By the discussion above, none of the
components can factor through a projective module over $C_{2^{n-1}}$
unless they are~0.
\end{proof}

We can readily prove separation when $s$ is even.

\begin{lemma} \label{la:sep}
For any even integer $s$ such that
$0 \le s \le 2^{n-1}$, the complex $K(V_{2^{n-1}+s}, V_s^2)$ of
$kG$-modules is separated.
\end{lemma}

\begin{proof}
Write $s = 2 s'$. From Lemma~\ref{la:tensind} we know that
$K(V_{2^{n-1}+s}, V_s^2) \cong K(V_{2^{n-2}+s'},
V_{s'}^2) \hspace{-0.1cm} \uparrow^{\otimes G}_H$. The right hand side
is separated by Lemma~\ref{la:tenindsep} and our induction
hypothesis.
\end{proof}

In view of this lemma, we assume now that $s$ is odd.

\begin{lemma} \label{la:s_in_c}
Let $s$ be an odd integer such that
$0 < s < 2^{n-1}$. Then (given our induction hypothesis):
\begin{enumerate}
\item[(a)] $\Lambda^r(V_{2^{n-1}-s}) \in c(G)$ for all $r \ge 0$ and

\item[(b)] $S^r(V_{2^{n-1}+s}) \in c(G)$ for all $0 \le r < 2^{n-1}$.
\end{enumerate}
Here $c(G)$ is the subgroup of the Green ring in Lemma~\ref{la:c}.
\end{lemma}

\begin{proof}
For part (a), the dimension $d=2^{n-1}-s$ of the module  is in the
range where we know that our formula for exterior powers (see
Theorem~\ref{thm:main} (d)) is valid by our continuing induction
hypothesis. The
statement is clearly true for $d=1$ and we can employ induction on
$d$, using the formula and the properties of $c(G)$ in
Lemma~\ref{la:c}.

For part (b) we use the formula $\sigma _t(V_{2^{n-1}+s}) =_{\proj}
\sigma _t (V_{2^{n}}) \lambda^{\Omega}_t (V_{2^{n-1}-s})$ from the
remark at the end of Section~\ref{sec:key} and part (a). The summands
of $S(V_{2^{n}})$ are permutation modules on a monomial basis, so are
in $c(G)$ unless the stabilizer of a monomial is of index 2.
But this first happens in degree $2^{n-1}$, because if a monomial fixed
by a subgroup of order $2^{n-1}$ contains~$y_i$, it must also contain
all $2^{n-1}$ elements of the orbit of $y_i$.
\end{proof}

%%%%%%%%%%%%%%%%%%%%%%%%%%%%%%%%%%%%%%%%%%%%%%%%%%%%%%%%%%%%%%%%%%%%%%%%%%%%%%%

\section{Separation}
\label{sec:sep2}

First we make some general constructions related to symmetric and exterior powers of vector spaces.
It is convenient to do this integrally first and then reduce modulo 2.
Let $U$ be a free module over the integers localized at $2$,
$\Z_{(2)}$. For $r \ge 0$ set
\[
T^r(U) = U \otimes _{\mathbb Z _{(2)}} \cdots \otimes _{\mathbb Z
  _{(2)}} U  \quad \mbox{($r$ times)}.
\]
Let the symmetric group $\Sigma _r$ act on $T^r(U)$ by permuting the
factors. Factoring out the action of $ \Sigma_r$ we get $S^r(U) =
T^r(U)/\Sigma_r=T^r(U) \otimes _{\Z_{(2)} \Sigma _r} \Z_{(2)}$. 
We can also let $\Sigma_r$ act on $T^r(U)$ by permuting the factors
and multiplying by the signature of the permutation, in which case we
write $T^r(U)_{\sigma}$. Similarly, on factoring out the action of
$\Sigma_r$ we obtain $\Lambda^r(U)=T^r(U)_{\sigma}/\Sigma_r$.

For any subset $I \subseteq \{1, \dots, r\}$ we set
$\Sigma_I := \{\pi \in \Sigma_r \mid \pi(i)=i \text{   for all   } i
\not \in I\}$ (so $\Sigma_I$ is a subgroup of $\Sigma_r$ isomorphic to
$\Sigma_{|I|}$). For $r \ge 2$, write $r=2^p+t$ with
$1 \leq t \leq 2^p$. Consider the subgroup
\[
Q_r := \begin{cases} \Sigma_{\{1,\dots,2^p\}} \times \Sigma_{\{2^p+1,\dots,r\}} & \mbox{if } t<2^p \\
                       (\Sigma_{\{1,\dots,2^p\}} \times
                       \Sigma_{\{2^p+1,\dots,2^{p+1}\}}) \rtimes \langle \tau
                       \rangle & \mbox{if } t=2^p \end{cases}
\]
of $\Sigma_r$, where $\tau \in \Sigma_r$ is the involution mapping $i$
to $2^p+i$ for $i=1,2,\dots,2^p$. The importance of~$Q_r$ lies the fact
that the index $|\Sigma _r : Q_r|$ is odd. This can been seen as
follows: $|\Sigma_{2^p+t}|/ (|\Sigma_{2^p}| \cdot |\Sigma_t|) =
\binom{2^p+t}{t}$, which is equal to the coefficient of $x^t$ in 
$(1+x)^{2^p+t} \equiv (1+x^{2^p})(1+x)^t \pmod{2}$; also
$(1+x)^{2^{p+1}} = (1+x^{2^p}+2X))^2 \equiv 1+2x^{2^p}+x^{2^{p+1}}
\pmod{4}$. 

Define  $L_S^r(U) := T^r(U)/Q_r$ and
$L_{\Lambda}^r(U) := T^r(U)_{\sigma}/Q_r$. There are natural quotient
maps $q_S: L_S^r(U) \rightarrow S ^r(U)$ and $q_{\Lambda}:
L_{\Lambda}^r(U) \rightarrow \Lambda^r(U)$, which have sections
$\tr_S : S^r(U) \rightarrow L_S^r(U)$ and
$\tr_{\Lambda} : \Lambda ^r(U) \rightarrow L_{\Lambda}^r(U)$ given by
$\tr x := \frac{1}{|\Sigma _r : Q_r|}\sum _{\pi \in \Sigma _r / Q_r} \pi x$.
These have the property that $q_S \circ \tr_S = \Id_{S^r(U)}$ and
$q_{\Lambda} \circ \tr_{\Lambda} = \Id_{\Lambda^r(U)}$. These maps are
all natural transformations of functors on free $\Z_{(2)}$-modules.

Writing $r=2^p+t$ as before, we see from the description of $Q_r$ that
\[
L_S^r(U) \cong \begin{cases} S^{2^p}(U) \otimes S^t(U) & \mbox{if } t<2^p \\
                               (S^{2^p}(U) \otimes S^{2^p}(U))/C_2 \cong
S^2(S^{2^p}(U)) & \mbox{if } t=2^p. \end{cases}
\]
Similarly, if $r \ge 3$ we have
\[
L_{\Lambda}^r(U) \cong \begin{cases} {\Lambda}^{2^p}(U) \otimes {\Lambda}^t(U) &
\mbox{if } t<2^p \\
                               ({\Lambda}^{2^p}(U) \otimes {\Lambda}^{2^p}(U))/C_2
\cong S^2(\Lambda^{2^p}(U)) & \mbox{if } t=2^p,
\end{cases}
\]
because the involution $\tau$ has signature 1, provided that $p \geq 1$.

Now let $V$ be an $\F_2$-vector space and let $U$ be a free
$\Z_{(2)}$-module such that $V \cong \F_2 \otimes_{\Z_{(2)}} U$.
Let $L^r$ denote one of the functors
$S^r,\Lambda^r,L_S^r,L_{\Lambda}^r$ above and use it to define a
functor with the same name on $\mathbb F_2$-vector spaces by
$L^r(V)=\mathbb F_2 \otimes _{\mathbb Z _{(2)}}L^r(U)$. This gives the
expected result for $S^r(V)$ and $\Lambda ^r(V)$.

In order to verify that $L^r$ is really a functor on vector spaces,
notice that if $U$ and $U'$ are two free $\mathbb Z _{(2)}$-modules
then the natural map $\Hom_{\mathbb Z _{(2)}} (U,U') \rightarrow \Hom
_{\mathbb F _2} (\mathbb F_2 \otimes _{\mathbb Z _{(2)}} U, \mathbb
F_2 \otimes _{\mathbb Z _{(2)}} U')$ is surjective, so all maps of
vector spaces lift. Furthermore, a map in the kernel has image in
$2U'$, so factors through multiplication by 2 on $U'$. But
multiplication by 2 on $U'$ induces multiplication by $2^r$ on
$T^r(U')_{\sigma}$, thus it induces 0 on
$\mathbb F_2 \otimes _{\mathbb Z _{(2)}}L^r(U')$.

It follows that the formulas above are also valid for
$\mathbb F_2$-vector spaces. A difference is that we now have natural
transformations $e^r:S^r \rightarrow \Lambda^r$ and
$L_e^r:L_{S}^r \rightarrow L_{\Lambda}^r$ induced by reducing modulo
squares.

The above functors induce functors on modules for a group in the obvious way.

\begin{remark}
Any representation of $G$ over a field of characteristic 2 can be written in
$\F_2$, so this is sufficient for our purposes. If we really needed functors
on vector spaces over a bigger field, this could be achieved by starting with a
larger ring than $\Z_{(2)}$.
\end{remark}

In the rest of this section we prove part
(a) of Theorem~\ref{thm:main}, assuming the whole of the theorem for
smaller $n$. We use the same notation as before.

\begin{lemma}
\label{la:L_ind}
Suppose that $V$ is a $kG$-module, $r \ge 2$ and
$L_e^r:L_S^r(V) \rightarrow L_{\Lambda}^r(V)$ is split injective modulo induced
summands. Then $e^r:S^r(V) \rightarrow \Lambda^r(V)$ is split injective modulo
induced summands.
\end{lemma}

\begin{proof}
Consider the commutative diagram
\begin{equation*}
\begin{CD}
S^r(V) @>{e^r}>> \Lambda ^r(V) \\
@V{\tr _S}VV    @VV{\tr _{\Lambda}}V \\
L_S^r(V) @>{L_e^r}>> L_{\Lambda}^r(V).
\end{CD}
\end{equation*}
The map $\tr_S$ is split injective, so if $L_e^r$ is split injective modulo induced
summands then so is $L_e^r \circ \tr _S$, by Lemma~\ref{la:prop_split}
(a). But this is equal to $\tr _{\Lambda} \circ \, e^r$ and
Lemma~\ref{la:prop_split} (b) shows that $e^r$ is split injective modulo induced
summands.
\end{proof}

\begin{lemma}
\label{la:Lsplit_mod_ind}
If $s$ is an odd integer such that
$0 < s < 2^{n-1}$ and $r$ is an integer such that
$0 \le r < 2^n$, then  $e^r:S^r(V_{2^{n-1}+s}) \rightarrow
{\Lambda}^r(V_{2^{n-1}+s})$ is split injective modulo induced summands.
\end{lemma}

\begin{proof}
We use induction on $r$. The cases $r=0,1$ are trivial and $r=2$ is
covered by Lemma~\ref{la:lambda2} combined with Lemma~\ref{la:sep_equiv}.
Let $r \geq 3$ and write $r=2^p+t$ with $1 \leq t \leq 2^p$.
Abbreviate $V_{2^{n-1}+s}$ to $V$. By Lemma~\ref{la:L_ind}, it is
sufficient to check that $L_e^r:L_S^r(V) \rightarrow L_{\Lambda}^r(V)$
is split injective modulo induced summands.

If $t < 2^p$ then $L_e^r=e^{2^p} \otimes e^t: S^{2^p}(V) \otimes S^{t}(V)
\rightarrow {\Lambda}^{2^p}(V) \otimes {\Lambda}^{t}(V)$. This is
split injective modulo induced summands by induction and
Lemma~\ref{la:prop_split} (c).

If $t=2^p$ then $L_e^r=S^2(e^{2^p}): S^2(S^{2^p}(V)) \rightarrow
S^2(\Lambda^{2^p}(V))$. By induction, $e^{2^p}$ is split injective modulo
induced summands, so it extends to a split injective map
$M:S^{2^p}(V) \rightarrow \Lambda^{2^p}(V) \oplus X$ with
left inverse~$N$, where $X$ is induced. By the remark after
Lemma~\ref{la:split_mod_ind}, we may assume that~$X$ only contains
summands that are also summands of $S^{2^p}(V)$; by
Lemma~\ref{la:s_in_c} and the assumption that $r < 2^n$ (and hence
$2^p < 2^{n-1}$, since $t=2^p$), these are of dimension divisible by 4.
Applying $S^2$, we see that $S^2(e^{2^p})$ extends to
\[
S^2(M):S^2(S^{2^p}(V)) \rightarrow S^2(\Lambda^{2^p}(V)) \oplus
S^2(X) \oplus (\Lambda^{2^p}(V) \otimes X),
\]
with left inverse $S^2(N)$. Certainly $\Lambda^{2^p}(V) \otimes X$ is
induced, and $S^2(X)$ is induced, by Corollary~\ref{cor:s2}.
Thus $L_e^r$ and thus $e^r$ are split injective modulo induced summands.
\end{proof}

Again, let $s$ be an odd integer such that
$0 < s < 2^{n-1}$. It follows from Lemmas~\ref{la:sep_in_posdeg},
\ref{la:sep_equiv} and \ref{la:Lsplit_mod_ind} that the complex
$K^r(V_{2^{n-1}+s}, V_s^2)$ is separated for all $0 \le r < 2^n$.

Recall that $K(V_{2^{n-1}+s}, V_s^2)$ is separated if and only if the
complex $L(V_{2^{n-1}+s}, V_s^2)$ from Section~\ref{sec:per} is
separated. For the rest of this section we will write just $K,L$
etc.\ . Now $L^r$ is separated for all $0 \le r < 2^n$, because it
coincides with $K^r$ in this range. We will show that $L^r$ is
separated for $r \ge 2^n$ by induction on $r$, so let $r \ge 2^n$ and
assume that the complex is separated in all lower degrees. 

By Lemma~\ref{la:sep_in_posdeg}, we can also assume that
$L^r$ is separated in positive (complex-)degrees,
so it is enough to prove that the short exact sequence
\begin{equation} \label{eq:sepseq}
0 \rightarrow \Ima(d_1^r) \rightarrow T^r \rightarrow
\widetilde{T}^r \rightarrow 0
\end{equation}
is separated at $T^r$.  By Lemma~\ref{la:dirsum}, the
restriction of $K$ to the maximal subgroup $H$
of $G$ decomposes as a tensor product of two complexes, and each of
these is separated, by our continuing induction hypothesis and
Theorem~\ref{thm:main}(a). Their product is also separated, by
Lemma~\ref{la:tensep}, hence so is $L^r$. It follows that the sequence
(\ref{eq:sepseq}) is separated at $T^r$ on
restriction to $H$.

But $T^r$ is induced for this range of $r$. Separation of
(\ref{eq:sepseq}) follows immediately from
Proposition~\ref{prop:homol2} applied to $\Ima(d_1^r) \rightarrow T^r$. 

This proves that the complex
$K^r(V_{2^{n-1}+s}, V_s^2)$ is separated for all $r \ge 0$, and part
(a) of Theorem~\ref{thm:main} follows.

%%%%%%%%%%%%%%%%%%%%%%%%%%%%%%%%%%%%%%%%%%%%%%%%%%%%%%%%%%%%%%%%%%%%%%%%%%%%%%%

\section{Exterior Powers}
\label{sec:ext}

 In this section we
prove part (d) of Theorem~\ref{thm:main}, assuming the whole of the
theorem for smaller~$n$.

Because we have already proved separation, periodicity and splitting
we know that
\[
\lambda _t(V_{2^{n-1}+s}) =_{\ind} \lambda _{t^2}^{\Omega}(V_{s})
  \lambda _t^{\Omega}(V_{2^{n-1}-s});
\]
see the first remark at the end of Section~\ref{sec:key}.
In order to obtain the formula with $=_{\proj}$, we first consider the
restriction to the subgroup $H$ of index 2. Writing $V_{2^{n-1}+s}
\downarrow ^G_H = V_{2^{n-2}+s'} \oplus V_{2^{n-2}+s''}$, the two
sides of the formula become
\[
\lambda _t(V_{2^{n-2}+s'})\lambda _t(V_{2^{n-2}+s''}) \quad \text{and}
\quad \lambda _{t^2}^{\Omega}(V_{s'})\lambda
_{t^2}^{\Omega}(V_{s''})\lambda _t^{\Omega}(V_{2^{n-2}-s'})\lambda
_t^{\Omega}(V_{2^{n-2}-s''}).
\]
But we know, by induction, that $\lambda _t(V_{2^{n-2}+s'}) =_{\proj}
\lambda _{t^2}^{\Omega}(V_{s'}) \lambda _t^{\Omega}(V_{2^{n-2}-s'})$
and similarly for $s''$. Thus, on restriction, the two sides are equal
modulo projectives. Now use Lemma~\ref{la:modind} in order to see that
the two sides are equal modulo projectives even before restriction.
This finally completes the proof of Theorem~\ref{thm:main}.

%%%%%%%%%%%%%%%%%%%%%%%%%%%%%%%%%%%%%%%%%%%%%%%%%%%%%%%%%%%%%%%%%%%%%%%%%%%%%%%%

\section{A Bound on the Number of Non-Induced Summands}
\label{sec:bound}

The description of the tensor product given in Section~\ref{sec:cyc2}
shows that the decomposition of $V_r \otimes V_s$ into indecomposable
summands involves a summand of odd dimension if and only both $r$ and
$s$ are odd, in which case it contains precisely one odd-dimensional
summand.

Let us write $\summ(V)$ for the number of indecomposable summands in
the module $V$.

\begin{proposition} \label{prop:summands}
The number of non-induced summands in $\Lambda(V)$ is at most
\[
2^{\frac{1}{2}(\dim (V) + \summ (V))}=\sqrt{2^{\summ(V)} \dim(\Lambda(V))}.
\]
\end{proposition}

\begin{proof}
Let $f(V)$ denote the number of non-induced summands in
$\Lambda(V)$. The comment about the tensor product above shows that
$f(V\oplus W)=f(V)f(W)$. The proposed bound also turns sums into
products, so it suffices to consider the case when $V=V_r$ is
indecomposable and show that $f(V_r) \leq 2^{\frac{1}{2}(r+1)}$.

We use induction on $r$. Since the cases $r = 0,1$ are trivial
we can assume that $r \ge 2$, and we can write $r = 2^{n-1}+s$, where
$1 \le s \le 2^{n-1}$. Setting $t=1$ in the formula
$\lambda _t(V_{2^{n-1}+s}) =_{\proj} \lambda _{t^2}^{\Omega}(V_{s})
\lambda _t^{\Omega}(V_{2^{n-1}-s})$ and using induction we obtain
\begin{eqnarray*}
f(\Lambda(V_{2^{n-1}+s})) & = & f(\lambda _1(V_{2^{n-1}+s})) =
f(\lambda_{1}^{\Omega}(V_{s}))f( \lambda _1^{\Omega}(V_{2^{n-1}-s}))=
f(\lambda_{1}(V_{s}))f( \lambda _1(V_{2^{n-1}-s})) \\
& \leq & 2^{\frac{1}{2}(2^{n-1}-s+1)} 2^{\frac{1}{2}(s+1)} =
2^{\frac{1}{2}(2^{n-1}+2)} \leq 2^{\frac{1}{2}(2^{n-1}+s+1)}.
\end{eqnarray*}
\end{proof}

For an indecomposable $kG$-module $V_r$, if we assume that the group
acts faithfully then $r > \frac{1}{2} |G|$, and the dimension of any
direct summand of $\Lambda(V_r)$ is at most $|G|$. It follows that the dimension of the
non-induced part of $\Lambda(V_r)$ divided by the dimension of the
whole of $\Lambda(V_r)$ is at most $2^{\frac{1}{2}(3-r)}r$.

%%%%%%%%%%%%%%%%%%%%%%%%%%%%%%%%%%%%%%%%%%%%%%%%%%%%%%%%%%%%%%%%%%%%%%%%%%%%%%%%

\section{Remarks}

\noindent (a) As already mentioned in the introduction, the formula in
Theorem~\ref{ext} reduces the computation of
$\Lambda^r(V_{2^{n-1}+s})$ to the computation of tensor products
of exterior powers of modules of smaller
dimension. Since tensor products can easily be determined recursively
(see Section~\ref{sec:cyc2}), this gives an efficient recursive method for
calculating the decomposition of exterior powers of modules for cyclic
2-groups into indecomposables. A program based
on this recurrence relation was implemented in GAP~\cite{GAP} by the
first author.

A restriction on the use of Theorem~\ref{ext} is the
growth of the multiplicities of direct summands of the form
$V_{2^m}$. For example, the multiplicity of $V_{128}$ as a direct
summand of $\Lambda^{57}(V_{147})$ is $8197519886357582844587268803532720$.
If one is only interested in the non-induced part of
$\Lambda^r(V_{2^{n-1}+s})$ the recurrence relation can be applied modulo
induced summands  to keep the multiplicities relatively small.

Together with the results in \cite{scyclic}, the
recurrence relation in Theorem~\ref{ext} also provides an algorithm
for computing the decomposition of the symmetric powers $S^r(V_m)$
into indecomposables for arbitrary $m$ and $r$.

\begin{example}
We determine the decomposition of $\Lambda^6(V_{13})$ into indecomposables:
\begin{eqnarray*}
\Lambda^6(V_{13}) & \cong_{\proj} & \Omega_{16}^{0+6}(\Lambda^0(V_5) \otimes \Lambda^6(V_3))
\oplus \Omega_{16}^{1+4}(\Lambda^1(V_5) \otimes \Lambda^4(V_3)) \oplus \\
&& \Omega_{16}^{2+2}(\Lambda^2(V_5) \otimes \Lambda^2(V_3))
\oplus \Omega_{16}^{3+0}(\Lambda^3(V_5) \otimes \Lambda^0(V_3)) \\
& \cong_{\proj} & (\Lambda^2(V_5) \otimes V_3) \oplus \Omega_{16}(\Lambda^3(V_5)).
\end{eqnarray*}
Furthermore $\Lambda^3(V_5) \cong \Lambda^2(V_5)$, by duality,
and
\[
\Lambda^2(V_{5}) \cong \Omega_{8}^{0+2}(\Lambda^0(V_1) \otimes \Lambda^2(V_3))
\oplus \Omega_{8}^{1+0}(\Lambda^1(V_1) \otimes \Lambda^0(V_3)) \cong
V_3 \oplus \Omega_8(V_1) \cong V_3 \oplus V_7.
\]
Thus
\begin{eqnarray*}
\Lambda^6(V_{13}) & \cong_{\proj} & (V_3 \oplus V_7) \otimes V_3 \oplus
\Omega_{16}(V_3 \oplus V_7) \cong_{\proj} (V_3 \otimes V_3) \oplus (V_7 \otimes V_3) \oplus
\Omega_{16}(V_3 \oplus V_7) \\
& \cong_{\proj} & (V_1 \oplus 2 V_4) \oplus (V_5 \oplus 2 V_8)
\oplus (V_{13} \oplus V_9) \cong_{\proj} V_1 \oplus 2 V_4 \oplus V_5 \oplus 2 V_8
\oplus V_{9} \oplus V_{13}.
\end{eqnarray*}
Comparing dimensions, we obtain
$\Lambda^6(V_{13}) \cong V_1 \oplus 2 V_4 \oplus V_5 \oplus 2 V_8
\oplus V_{9} \oplus V_{13} \oplus 104 V_{16}$.
\end{example}

\medskip

\noindent (b) Obviously, Gow and Laffey's formula for exterior
squares~\cite[Theorem~2]{gowlaffey} is the special case $r=2$
of Theorem~\ref{ext}. Furthermore, setting $s=2^{n-1}-1$ in
Theorem~\ref{ext} gives Kouwenhoven's formula
\cite[Theorem~3.4]{kouwenhoven} for $\Lambda^r(V_{q-1})$ when $q$ is a
power of $2$ (for all $r$).

In \cite[Theorem~3.5]{kouwenhoven} Kouwenhoven proved the
formula
\begin{equation} \label{eq:qplus1}
\lambda_t(V_{q+1}-V_{q-1}) = 1 + (V_{q+1}-V_{q-1})t + t^2
\end{equation}
in $a(C_{q \cdot p})[[t]]$, where $q$ is a power of a prime $p$;
see Section~\ref{sec:key} for a definition of $\lambda_t$. We
will show how this can be derived from
Theorem~\ref{ext} in the case that $p=2$. Note that, since the
dimension series of the two sides match, it is sufficient to prove
this modulo projectives. The theorem gives us: 
\[
\lambda_t(V_{2^{n-1}+1}) =(1+V_{2^n-1}t^2)\lambda^{\Omega_{2^n}}_t(V_{2^{n-1}-1})
\] modulo $V_{2^n}$ and
\[
\lambda_t(V_{2^{n-1}-1}) = (1+V_{2^{n-1}-1}t)\lambda^{\Omega_{2^{n-1}}}_{t^2}(V_{2^{n-2}-1})
\]
modulo $V_{2^{n-1}}$. The latter can be written as
\[
\lambda_t(V_{2^{n-1}-1}) =
(1+V_{2^{n-1}-1}t)(\lambda^{\Omega_{2^{n-1}}}_{t^2}(V_{2^{n-2}-1})+V_{2^{n-1}}f(t)) 
\]
exactly (the last term can be written inside the parentheses, since
$(1+V_{2^{n-1}-1}t)$ is invertible). Applying $\Omega_{2^n}$ in odd
degrees we obtain 
\[
\lambda_t^{\Omega_{2^n}}(V_{2^{n-1}-1}) =
(1+V_{2^{n-1}+1}t)(\lambda^{\Omega_{2^{n-1}}}_{t^2}(V_{2^{n-2}-1})+V_{2^{n-1}}f(t)). 
\]
Substituting into the left hand side of (\ref{eq:qplus1}) yields
\[
(1+V_{2^n-1}t^2)(1+V_{2^{n-1}+1}t)(1+V_{2^{n-1}-1}t)^{-1}
\]
modulo $V_{2^n}$, and it is easy to verify that
\[
(1+V_{2^n-1}t^2)(1+V_{2^{n-1}+1}t)=(1+V_{2^{n-1}-1}t)(1 +
(V_{2^{n-1}+1}-V_{2^{n-1}-1})t + t^2) 
\]
modulo $V_{2^n}$.

\medskip

\noindent (c) Theorem~\ref{ext} can also be used to calculate the
Adams operations on the Green ring $a(C_{2^n})$, as was shown to us by
Roger Bryant and Marianne Johnson.  For each $r>0$ and $j \in
\{1,\dots, 2^n\}$, 
define an element $\psi^r_\Lambda(V_j) \in a(C_{2^n})$ by
\[
\psi_\Lambda^1(V_j) - \psi_\Lambda^2(V_j)t + \psi_\Lambda^3(V_j)t^2 - \cdots = \psi_{\Lambda,t}(V_j)
= {\textstyle \frac{d}{dt}} \log \lambda_t(V_j).
\]
By $\Z$-linear extension there is a map
$\psi^r_\Lambda: a(C_{2^n}) \rightarrow a(C_{2^n})$, called the $r$th
Adams operation defined by the exterior powers. It can
be shown that if~$r$ is odd, then $\psi^r_\Lambda$ is the identity map
on $a(C_{2^n})$ and that $\psi^{2^i r}_\Lambda = \psi^{2^i}_\Lambda$ for all
$i \ge 1$, so all that remains is to describe~$\psi^{2^i}_\Lambda$
for $i \ge 1$ (see~\cite{bryantjohnson1} and \cite{bryantjohnson2} for
details). For $j \geq 2$, write $j=2^m+s$ with $m \ge 0$ and
$1 \le s \le 2^m$; then
\[
\psi^{2^i}_\Lambda(V_{2^m+s}) = 2 \psi^{2^{i-1}}_\Lambda(V_s) +
\psi^{2^i}_\Lambda(V_{2^m-s})
\]
for all $i \ge 2$ and
\[
\psi^2_\Lambda(V_{2^m+s}) = 2V_{2^{m+1}} - 2V_{2^{m+1}-s} +
\psi^2_\Lambda(V_{2^m-s}).
\]
This can be seen by applying the definition of the Adams operations to
the Hilbert series form of Theorem~\ref{ext}, obtaining (in the
obvious notation) 
\[
\psi_{\Lambda,t}(V_{2^m+s})=_{\proj} 2t \psi^{\Omega}_{\Lambda,t^2}(V_s)  + \psi^{\Omega}_{\Lambda,t}(V_{2^m-s}).
\]

%%%%%%%%%%%%%%%%%%%%%%%%%%%%%%%%%%%%%%%%%%%%%%%%%%%%%%%%%%%%%%%%%%%%%%%%%%%%%%%%

\end{document}